\documentclass[11pt]{amsart}
\usepackage{amsfonts,amssymb,amscd,amsmath,enumerate,verbatim,calc, amsthm, amsxtra}
\usepackage[colorlinks,linkcolor=blue,pagebackref=true,citecolor=red]{hyperref}
\usepackage{pgfplots}
\usepackage{extarrows}
\usepackage{blkarray}
\usepackage{setspace}
\makeatletter
\newcommand{\nextverbatimspread}[1]{%
  \def\Verbatim@font{%
    \linespread{0cm}\normalfont\ttfamily
    \gdef\Verbatim@font{\normalfont\ttfamily}}
}
\makeatother
\usepackage{fancyvrb}
\usepackage[all]{xy}
\usepackage{tikz}
\usetikzlibrary{matrix}
 \usepackage[mathscr]{euscript} 
 \usepackage{ mathrsfs }
 \usepackage[left=2.25cm,right=2.25cm,top=2.5cm,bottom=2.5cm]{geometry}

\usepackage{graphicx}
\usepackage{color}

\newcommand{\CM}{Cohen-Macaulay}

\newcommand{\m}{\mathfrak{m} }

\newcommand{\R}{\mathcal{R}}

\providecommand{\D}{{\mathcal D}}
\newcommand{\Z}{\mathbb{Z} }
\newcommand{\N}{\mathbb{N} }
\newcommand{\TT}{\mathcal{T}}

\newcommand{\K}{\mathbb{K}_{\bullet} }
\newcommand{\A}{\mathcal{A}}
\newcommand{\B}{\mathcal{B}}

\newcommand{\F}{\mathbf{F}_{\bullet} }

\newcommand{\lrt}{\longrightarrow}

\newcommand{\bideg}{\operatorname{bideg}}
\newcommand{\im}{\operatorname{image}}
\newcommand{\coker}{\operatorname{coker}}

\newcommand{\sym}{\operatorname{Sym}}
\newcommand{\chars}{\operatorname{char}}
\newcommand{\depth}{\operatorname{depth}}
\newcommand{\grade}{\operatorname{grade}}

\newcommand{\gr}{\operatorname{gr}}

\newcommand{\fitt}{\operatorname{Fitt}}

\newcommand{\reg}{\operatorname{reg}}
\newcommand{\sol}{\operatorname{Sol}}
\newcommand{\reltype}{\operatorname{reltype}}

\newcommand{\inn}{\operatorname{in}}

\newcommand{\hgt}{\operatorname{ht}}

\providecommand\Spec{\text{\rm Spec}}
\newcommand{\projdim}{\operatorname{projdim}}

\newcommand{\Hom}{\operatorname{Hom}}
\newcommand{\Ext}{\operatorname{Ext}}
\newcommand{\Tor}{\operatorname{Tor}}

\newcommand{\lcm}{\operatorname{LCM}}

\theoremstyle{plain}

\newtheorem{thm}{Theorem}
\newtheorem{theorem}{Theorem}[section]
\newtheorem{corollary}[theorem]{Corollary}
\newtheorem{lemma}[theorem]{Lemma}
\newtheorem{proposition}[theorem]{Proposition}

\theoremstyle{definition}
\newtheorem{definition}[theorem]{Definition}
\newtheorem{s}[theorem]{}
\newtheorem{remark}[theorem]{Remark}

\newtheorem{example}[theorem]{Example}
\newtheorem{notation}[theorem]{Notation}
\newtheorem*{example*}{\it Example}

\theoremstyle{remark}
\newtheorem{claim}{\it Claim}
\newtheorem*{claim*}{\it Claim}

\newtheorem*{case*}{\it Case}
\newtheorem*{note*}{\it Note}

\begin{document}
\title[On the defining equations of Rees algebra]{On the defining equations of Rees algebra of a height two perfect ideal using the theory of $D$-modules}
\author{Sudeshna Roy}
\date{\today}
\address{Department of Mathematics, IIT Bombay, Powai, Mumbai 400 076}
\email{sudeshnaroy.1@gmail.com}

\begin{abstract}
Let $k$ be a field of characteristic zero, and $R=k[x_1, \ldots, x_d]$ with $d \geq 3$ be a polynomial ring in $d$ variables. Let $\m=(x_1, \ldots, x_d)$ be the homogeneous maximal ideal of $R$. Let $\mathcal{K}$ be the kernel of the canonical map $\alpha: \sym(I) \rightarrow \R(I)$, where $\sym(I)$ (resp. $\R(I)$) denotes the symmetric algebra (resp. the Rees algebra) of an ideal $I$ in $R$. We study $\mathcal{K}$ when $I$ is a height two perfect ideal minimally generated by $d+1$ homogeneous elements of same degree and satisfies $G_d$, that is, the minimal number of generators of the ideal $I_{\mathfrak{p}}$, $\mu(I_{\mathfrak{p}}) \leq \dim R_{\mathfrak{p}}$ for every $\mathfrak{p} \in V(I) \backslash \{\m\}$. We show that 
\begin{enumerate}[{\rm (i)}]
	\item $\mathcal{K}$ can be described as the solution set of a system of differential equations,
	\item the whole bigraded structure of $\mathcal{K}$ is characterized by the integral roots of certain $b$-functions,
	\item certain de Rham cohomology groups can give partial information about $\mathcal{K}$. 
\end{enumerate}
\end{abstract}
\maketitle

\section{Introduction}
Let $R$ be a commutative Noetherian ring, and $I=(f_1, \ldots, f_n)$ be an ideal in $R$. The {\it Rees algebra} of I is defined as $\R(I) = R[It] = \bigoplus_{i=0}^ \infty I^it^i$. We can see $\R(I)$ as a quotient of the
polynomial ring $S=R[T_1,\ldots,T_{n}]$ via the surjective map 
\begin{equation}\label{intro-eq1}
	\begin{split}
	\psi: & S \to\R(I) \\
	& T_i \mapsto f_it.	
	\end{split}
\end{equation}
The {\it defining ideal} of the Rees algebra $\R(I)$ is $\mathcal{I} = \ker\psi$, the kernel of the map $\psi$ and the {\it defining equations} are the generators 
of $\mathcal{I}$. Now the map $\theta: R^n \to I$ defined by $(r_1, \ldots, r_n) \mapsto \sum_{i=1}^n r_i f_i$ induces an $R$-algebra homomorphism $\beta: S=R[T_1, \ldots, T_n] \to \sym(I)$, where by $\sym(I)$ we denote the {\it symmetric algebra} of $I$. Clearly $\mathcal{L}:=\ker \beta$ is generated by all linear forms $\sum_{i=1}^n r_i T_i$ such that $\sum_{i=1}^n r_i f_i=0$. Therefore $\psi$ factors through $\sym(I)$ and we get the following diagram 
\[
\xymatrix@R0.5cm{S \ar[d]^{\beta} \ar[r]^{\psi}& \R(I)\ar@{<-}[ld]^{\alpha}\\
	\sym(I)}\] 
Thus we get a relation between $\sym(I)$ and $\R(I)$ in the form of the following exact sequence 
\[0 \to \mathcal{K} \to \sym(I) \overset{\alpha}{\lrt} \R(I) \to 0,\]
where $\mathcal{K}=\ker \alpha$. As $\mathcal{K}\cong \mathcal{I}/\mathcal{L}$, 
it is enough to study is $\mathcal{K}$. Determining the equations of the Rees algebra is a fundamental problem and several works are going on this topic. A brief survey 
in this regards can be found in \cite[Introduction]{KPU}. 
The defining ideal is known when $I$ is a linearly presented height two perfect ideal satisfying $G_d$. In this case $\R(I)$
is \CM. So in recent years, our main aim is to investigate $\mathcal{I}$ when $I$ is a height two perfect ideal which is either not linearly presented or does not satisfy $G_d$. In this situation mostly $\R(I)$ is not \CM. 
A study of Rees algebra of grade two, almost linearly presented (all but the last column of the presentation matrix $\varphi$ of the ideal are linear and the last column consists of homogeneous entries of arbitrary degree $n \geq 1$), perfect ideal which is minimally generated by homogeneous elements of the same degree, was done by A. R. Kustin, C. Polini and B. Ulrich in \cite{KPU} for $d=2$ and was generalized by J. A. Boswell and V. Mukundan in \cite{BM} for $d>2$.
In \cite{YC}, Yairon Cid first described $\mathcal{K}$ when $I$ is a height two ideal in $k[x_1,x_2]$, minimally generated by three homogeneous polynomials of the same degree using ``$D$-module theory". 
In this article we generalize his results for dimension $d \geq 3$ mostly following his path. We restrict our attention to height two perfect ideals satisfying $G_d$ and minimally generated by $d+1$ homogeneous elements of same degree. 

\s\label{Std-Ass} Let $k$ be a field of characteristic zero, $R=k[x_1, \ldots, x_d]$ with $d \geq 3$ be a polynomial ring in $d$ variables, and $\m=(x_1, \ldots, x_d)$ be the homogeneous maximal ideal in $R$. Let $I=(f_1,\ldots,f_{d+1}) \subset R$ be a perfect ideal of height two minimally generated by $d+1$ homogeneous elements of same degree $\nu$. Notice $\R(I) \cong S/ \mathcal{I}$ where $S=R[T_1,\ldots,T_{d+1}]=k[x_1, \ldots, x_d,T_1,\ldots,T_{d+1}]$. We consider $S$ as a bigraded $k$-algebra, where $\bideg(T_i)=(1, 0)$ and $\bideg (x_j)=(0, 1)$. Further, if we consider $\bideg t=(1, -\nu)$, then the map $\psi$ defined in \eqref{intro-eq1} becomes bi-homogeneous and hence $\mathcal{I}=\ker \psi$ becomes a bi-graded $S$-module. Let $U=k[T_1,\ldots,T_{d+1}]$ be a polynomial ring. Clearly $S=R \otimes_k U$. If $M$ is a bi-graded $S$-module, we write \[M_{p,*}=\bigoplus_{q \in \Z}M_{p,q} \quad \text{and} \quad M_{*,q}=\bigoplus_{p \in \Z}M_{p,q},\] where $M_{p,*}$ is a graded $R$-module and $M_{*,q}$ is a graded $U$-module. 

As $I$ is a perfect ideal, $\projdim_R(R/I)=\grade I=\hgt I =2$ 
and hence $\projdim I=1$. By {\it Hilbert-Burch theorem}, $I$ has a free resolution of the form
\begin{equation}\label{eq1}
0 \to R^{d} \overset{\varphi} \longrightarrow R^{d+1} \xlongrightarrow{[f_1, \ldots, f_{d+1}]} I \to 0
\end{equation}
 and $I \cong I_{d}(\varphi)$ where $I_{d}(\phi)$ is the ideal generated by $d \times d$-minors of $\varphi=(a_{i,j})_{(d+1) \times d}$ (see \cite[Theorem 1.4.17]{BH}).
If
 \begin{equation}\label{symeq}
 [g_1, \ldots, g_d]= [T_1, \ldots, T_{d+1}] \cdot \varphi,
 \end{equation}
then $\sym(I) \cong S/(g_1, \ldots, g_d)$. Thus $\mathcal{K} = \mathcal{I}/(g_1, 
\ldots, g_d)$. Note that 
\[g_j=a_{1,j}T_1+ \cdots+a_{d+1,j}T_{d+1} \quad \forall 1\leq j\leq d.\] Set $\deg a_{i,j}=\nu_j$ for all $1 \leq j \leq d$ to make all $g_j$ homogeneous with $\bideg g_j=(1,\nu_j)$ in $S$. Then $\mathcal{K}$ gets a natural structure of bi-graded $S$-module.  Since $I$ is generated by the maximal minors of $\varphi$ so we get that $\sum_{i=1}^{d}\nu_i = \nu$. We can rewrite \eqref{eq1} in the following way 
\begin{equation}\label{eq2}
0 \to \bigoplus_{i=1}^{d} R(-\nu-\nu_i) \overset{\varphi} \longrightarrow \bigoplus_{i=1}^{d+1} R(-\nu) \xlongrightarrow{[f_1,\ldots, f_{d+1}]} I \to 0,
\end{equation} For $1 \leq r \leq d$, let $\varphi_r$ be the $(d+1) \times r$ matrix consisting of the first $r$ columns of $\varphi$ and set
$E_r = \coker \varphi_r$, so that $\varphi_{d}=\varphi$ and $E_{d} = I$. As $\varphi$ is injective so is $\varphi_r$ for all $1 \leq r \leq d$. Note that
\begin{equation}\label{eq3}
	0 \to \bigoplus_{i=1}^{r} R(-\nu-\nu_i) \overset{\varphi_r} \longrightarrow \bigoplus_{i=1}^{d+1} R(-\nu) \xlongrightarrow{[f_1,\ldots, f_{d+1}]} E_r \to 0
\end{equation}
is a presentation of $E_r$. Now $(g_1, \ldots, g_r)=[T_1, \ldots, T_{d+1}] \cdot \varphi_r$ implies that $\sym(E_r) \cong S/(g_1, \ldots, g_r)$. 
\begin{definition}
If the minimal number of generators of the ideal $I_{\mathfrak{p}}$, $\mu(I_{\mathfrak{p}}) \leq \dim R_{\mathfrak{p}}$ for every $\mathfrak{p} \in V(I) \backslash \{\m\}$, then we say that $I$ satisfies $G_d$, where $d$ is the dimension of the ring $R$.
\end{definition}
In our case, as $I$ is a height two perfect ideal generated by homogeneous polynomials of the same degree $\nu$, by \cite[(3.7.3)]{KPU} \[\mathcal{K'}= H^0_{\m}(\sym(I(\nu))) \iff I \text{ satisfies } G_d,\] where $\mathcal{K}'=\ker\left(\sym(I(\nu)) \to \R(I)\right)$. Moreover, as $\mu(I)=d+1 > \dim R=d$ so $\mathcal{K}' \neq 0$ by \cite[(3.7.2)]{KPU}.  

\s\label{equv-Gd-Fitt} As mentioned in \cite[p. 2]{KPU}, the condition $G_d$ can be interpreted in terms of the height of Fitting ideals.
If $I$ has a presentation \eqref{eq1}, then $I$ satisfies $G_d$ if $\hgt \fitt_i(I)=\hgt I_{d+1-i}(\varphi)>i$ for all $i<d$, that is, $\hgt I_{i}(\varphi)>d+1-i$ for all $d+1-i<d$, i.e., $i>1$. Thus $\hgt I_{2}(\varphi)\geq d$ (as $\hgt I_{2}(\varphi)>d+1-2=d-1$). Since $I_1(\varphi) \supseteq I_2(\varphi)$ so we get $d \leq \hgt I_{2}(\varphi) \leq \hgt I_1(\varphi) \leq \dim R=d$ and hence $\hgt I_1(\varphi)=d$. 

\begin{remark}
Let $I$ be a height two ideal in $k[x_1,x_2]$, minimally generated by three homogeneous polynomials of the same degree. By {\it Auslander Buchsbaum formula} we have $\projdim_R(R/I)+\depth(R/I)=\depth(R)$. Since $\dim R/I=\dim R -\hgt I=0$ so $\depth(R/I)=0$. Thus $\projdim_R(R/I)=2$ which implies that $\projdim_R(I)=1$. By \cite[Proposition 1.4.5]{BH} we get that $I$ has a free resolution \[\F: \quad 0\to R^2 \overset{\varphi}{\lrt} R^3 \to I \to 0.\] So $I \cong I_2(\varphi)$ is a perfect ideal by {\it Hilbert-Burch theorem}.
As $\hgt I_2(\varphi)=\hgt I=2$, by \ref{equv-Gd-Fitt} it follows that $I$ satisfies $G_2$.	
\end{remark}

This article contains the following four main results which are generalization of Theorem A, Theorem B, Theorem C and Theorem D in \cite{YC}. Let $D=A_d(k)=k\langle x_1, \ldots,x_d, \partial_1, \ldots, \partial_d\rangle$ be the $d$-th {\it Weyl algebra} over $k$ and $\mathcal{T}=A_d(k)[T_1, \ldots, T_{d+1}]$ be a polynomial ring over $D$. We define $d$ differential operators $L_i=\mathcal{F}(g_i)$ for all $1 \leq i \leq d$ by applying the Fourier transform $\mathcal{F}$ to $g_i$'s in \eqref{symeq}. Here $\mathcal{F}: \mathcal{T} \to \mathcal{T}$ is an automorphism defined by $\mathcal{F}(x_i)=\partial_i, \mathcal{F}(\partial_i)=-x_i$ and $\mathcal{F}(T_i)=T_i$. The following result says that $\mathcal{K}$ can be described as a solution set of a system of differential equations. As an application of it we get Corollary \ref{vanishing-nonvanishing} which gives the highest $x$-degree for an element in the graded part $\mathcal{K}_{p,*}$ for all $p \geq d$. 
\begin{thm}
	\label{thmA}
	Let $I \subset R=k[x_1,\ldots, x_d]$ with $d \geq 3$ be a perfect ideal of height two satisfying $G_d$ and minimally generated by $d+1$ homogeneous elements of same degree $\nu$, and let $L_i=\mathcal{F}(g_i)$ be the Fourier transform of $g_i$ from \eqref{symeq} with $\bideg g_i=(1,\nu_i)$ in $S=R[T_1,\ldots,T_{d+1}]$ for $i=1, \ldots, d$. Then we have the following isomorphism of bigraded $S$-modules 
	$$
	\mathcal{K} \cong {\sol\Big(L_1,\ldots, L_d;S\Big)}_{\mathcal{F}}(-d, -\nu+d), 
	$$ 
	where ${\sol(L_1,\ldots, L_d;S)}=\{ h \in S \mid L_i \bullet h = 0 \text{ for all } i=1, \ldots, d\}$ and the subscript-$\mathcal{F}$ is used to stress the bigraded $S$-module structure induced on $\sol(L_1,\ldots,L_d;S)$ by the twisting of the Fourier transform ${\mathcal{F}}$ {\rm (}see Section 3 or \cite[Lemma 3.9]{YC}{\rm )}.
\end{thm}

The following result gives us information about 
the lowest possible $x$-degree for an element in the graded part $\mathcal{K}_{p,*}$ and describes the $b$-function of a family of holonomic $D$-modules which is defined in Notation \ref{not}.
\begin{thm}[with hypotheses as in Theorem \ref{thmA}]
	\label{thmB}
	Then for each integer $p \ge d$ there exists a nonzero $b$-function $b_p(s)$, and we have a relation between the graded structure of $\mathcal{K}_{p,*}$ and the integral roots of $b_p(s)$ given in the following equivalence
	\[\mathcal{K}_{p,u} \neq 0 \qquad \Longleftrightarrow \qquad b_p(-\nu+d+u)=0.\]
	Even more, we have that these are the only possible roots of $b_p(s)$, that is,
	\[b_p(s)=\prod_{\{u \in \mathbb{Z}\;\mid\; \mathcal{K}_{p,u}\neq 0\}} (s+\nu-d-u).\]
\end{thm}
The following result shows that there is an isomorphism of graded $U$-modules between $\mathcal{K}$ and a certain de Rham cohomology group.
\begin{thm}[with hypotheses as in Theorem \ref{thmA}]
	\label{thmC}
	Then we have the following isomorphism of graded $U=k[T_1,\ldots,T_{d+1}]$-modules
	\[\mathcal{K} \cong H_{dR}^0(Q) = \{w \in Q \mid \partial_i \bullet w= 0 \text{ for all } i=1, \ldots, d \},\]
	where $Q$ denotes the left $\TT$-module $\TT/\TT(L_1,\ldots, L_d)$. In particular, for any integer $p$ we have an isomorphism of $k$-vector spaces
\[\mathcal{K}_{p,*} \cong H_{dR}^0(Q_p) = \{w \in Q_p \mid \partial_i \bullet w= 0 \text{ for all } i=1, \ldots, d \}.\]
\end{thm}

\begin{thm}[with hypotheses as in Theorem \ref{thmA}]
	\label{thmD}
	Then we have the following isomorphism of bigraded $S$-modules
\[\mathcal{K} \cong \Big\{ w \in H^2_{\m}\left(\frac{S}{(g_1,\ldots,g_{d-2})}\right)(-2, -\nu_d-\nu_{d-1}) \mid g_{d-1} \cdot w=0 \text{ and } g_d \cdot w=0\Big\}.\]
\end{thm}

\begin{remark}
We put the extra conditions on $I$ (along with $\hgt I=2$ and $\mu(I)=d+1$), that $I$ is a perfect ideal satisfying $G_d$, 
mainly to prove Theorem \ref{thmD}. Although this result does not look promising, but surprisingly using this we get Theorem \ref{thmA}. Under these extra conditions $g_1, \ldots, g_d$ is a regular sequence in $S$, see Lemma \ref{reg}. This fact helps us to prove Proposition \ref{koszul-d(Q)} which is used to prove Theorem \ref{thmB} and Theorem \ref{thmC}. 
Rest of the results follow by almost similar methods used in \cite{YC}. We keep the name of the sections, statements of the results and all other arrangements as it is in \cite{YC} so that readers can easily do a comparative study.

The article is organized as follows. In Section 2 we prove Theorem \ref{thmD}, in Section 3 we prove Theorem \ref{thmA}, in Section 4 we prove Theorem \ref{thmB} and in Section 3 we prove Theorem \ref{thmC}. In Section 6, we generalize the function given in \cite[Section 6]{YC} that can compute the $b$-function $b_p(s)$ in {\it Macaulay2} and give some examples to show that how it helps us to recover the bi-graded structure of $\mathcal{K}$. In the last section we show that the generalized version can help us to compute $\reltype F(I), \reg F(I), e(F(I))$ and $r(I)$, where $\reltype F(I)$ denotes the relation type of $F(I)$ (see Definition \ref{reltype}), $\reg F(I)$ denotes the regularity of $F(I)$, $e(F(I))$ denotes the Hilbert-Samuel multiplicity of $F(I)$ and $r(I)$ denotes the reduction number of $I$. Using this we can also get a lower bound of $\reltype \R(I)$ and hence $\reg \R(I)$.
\end{remark}

\section{An ``explicit" description of the equations}

\s\label{ass} {\bf Assumptions:} We will prove our results under the following assumptions: 

(i) $R=k[x_1, \ldots, x_d]$ with $\chars k=0$ and $d \geq 3$, 

(ii) $I$ is a perfect ideal of height $2$, satisfies $G_d$ and $I=(f_1,\ldots,f_{d+1})$ with $\deg f_i=\nu$ for all $i$.

\vspace{0.25cm}
Then we have observed in \ref{Std-Ass}, $I$ has a presentation \eqref{eq2}, $\sym(I) \cong S/(g_1, \ldots, g_d)$ where $g_1, \ldots, g_d$ as in \eqref{symeq}, and $S=R[T_1, \ldots, T_{d+1}]$ is a bigraded $k$-algebra with $\bideg(T_i)=(1, 0), ~\bideg (x_j)=(0, 1)$.
 
 \begin{lemma}[with hypotheses as in \ref{ass}]\label{reg}
$g_1,\ldots,g_d$ is a regular sequence in $S$.
  \end{lemma}

 \begin{proof}
 By Huneke and Rossi's result (see \cite{HR}) we have \[\dim \sym(I)=\max_{\mathfrak{p} \in \Spec(R)} \{\dim R/\mathfrak{p}+ \mu(I_{\mathfrak{p}})\}.\] Since $I$ satisfies $G_d$ so for any $\mathfrak{p} \neq \m$ we have $\mu(I_{\mathfrak{p}}) \leq \dim R_{\mathfrak{p}}=\hgt \mathfrak{p}$ and hence $\dim R/\mathfrak{p}+ \mu(I_{\mathfrak{p}}) \leq \dim R/\mathfrak{p}+ \hgt \mathfrak{p}=\dim R=d$ (as $R$ is \CM). On the other hand, $\dim R/\m+ \mu(I_{\m})=0+d+1=d+1$. It follows that $\dim \sym(I)=d+1$. Therefore $\hgt (g_1, \ldots, g_d)=d$ (as $S$ is \CM). Hence $g_1,\ldots,g_d$ is a regular sequence in $S$ (see \cite[Theorem 17.4]{Mat}).
 \end{proof}

\s{\bf Generalization of \cite[Lemma 2.3]{YC}}:
We have a short exact sequence \[0 \to S(-1,-\nu_1) \overset{g_1}{\lrt} S \to S/(g_1) \to 0 \] which induces a long exact sequence 
\begin{align*}
\xymatrix@R0.25cm{
0 \ar[r]& H_{\m}^0(S)(-1,-\nu_1) \ar[r]^-{g_1}& H_{\m}^0(S) \ar[r]& H_{\m}^0(S/(g_1)) \\
 && \vdots &\\
  \ar[r]& H_{\m}^{d-1}(S)(-1,-\nu_1) \ar[r]^-{g_1}& H_{\m}^{d-1}(S) \ar[r]& H_{\m}^{d-1}(S/(g_1)) \\ 
 \ar[r]& H_{\m}^d(S)(-1,-\nu_1) \ar[r]^-{g_1}& H_{\m}^d(S)\ar[r]& H_{\m}^d(S/(g_1))\\
 \ar[r]& H_{\m}^{d+1}(S)(-1,-\nu_1) \ar[r]^-{g_1}& H_{\m}^{d+1}(S) \ar[r]& \cdots
}
\end{align*}
Since $R \subseteq S$ is a flat ring extension so we get that
 \begin{equation*}
H^j_{\m S}(S) \begin{cases}
\neq 0 & \text{if } j=d \\
=0  & \text{otherwise},
\end{cases}
\end{equation*}
using the fact that 
 \begin{equation*}
 H^j_{\m}(R) = \begin{cases}
 x_1^{-1}\cdots x_d^{-1}k[x_1^{-1},\ldots,x_d^{-1}]  & \text{if } j=d \\
 0  & \text{otherwise},
 \end{cases}
 \end{equation*}
Thus $H_{\m}^{i}(S/(g_1))=0$ for all $1 \leq i \leq d-2$, and we get an exact sequence 
\begin{equation}\label{e0}
0 \to H_{\m}^{d-1}(S/(g_1)) \overset{\gamma_1}{\lrt} H_{\m}^d(S)(-1,-\nu_1) \overset{g_1}{\lrt} H_{\m}^d(S) \to H_{\m}^d(S/(g_1)) \to 0.
\end{equation} 
Since $g_1, \ldots, g_d$ is a regular sequence so we have short exact sequences 

\begingroup
\allowdisplaybreaks
\begin{align*}
\xymatrix@R0.25cm{
 0 \ar[r]& \frac{S}{(g_1)}(-1,-\nu_2) \ar[r]^-{g_2}& \frac{S}{(g_1)} \ar[r]& \frac{S}{(g_1,g_2)}=\sym(E_2) \ar[r]& 0\\
  0 \ar[r]& \frac{S}{(g_1,g_2)}(-1,-\nu_3) \ar[r]^-{g_3}& \frac{S}{(g_1,g_2)} \ar[r]& \frac{S}{(g_1,g_2,g_3)}=\sym(E_3) \ar[r]& 0\\
  &&\vdots&&\\
   0 \ar[r]& \frac{S}{(g_1,\ldots,g_{d-1})}(-1,-\nu_d) \ar[r]^-{g_d}& \frac{S}{(g_1,\ldots,g_{d-1})} \ar[r]& \frac{S}{(g_1,\ldots,g_{d})}=\sym(I) \ar[r]& 0
}
\end{align*} 
\endgroup

From the induced long exact sequences we get \[H^i_{\m}(S/(g_1,\ldots,g_j))=0 \text{ for all } 0 \leq i \leq d-(j+1) \text{ and } 1 \leq j \leq d-1\] 
and the following exact sequences 
\begingroup
\allowdisplaybreaks
\begin{align}\label{partial-g-rel}
\xymatrix@R0.25cm{
	0 \ar[r]& H^{d-2}_{\m}\left(\frac{S}{(g_1,g_2)}\right) \ar[r]^-{\gamma_2}& H^{d-1}_{\m}\left(\frac{S}{(g_1)}\right)(-1,-\nu_2) \ar[r]^-{g_2}& H^{d-1}_{\m}(\frac{S}{(g_1)}), \\
0 \ar[r]& H^{d-3}_{\m}\left(\frac{S}{(g_1,g_2,g_3)}\right) \ar[r]^-{\gamma_3}& H^{d-2}_{\m}\left(\frac{S}{(g_1,g_2)}\right)(-1,-\nu_3) \ar[r]^-{g_3}& H^{d-2}_{\m}\left(\frac{S}{(g_1,g_2)}\right), \\
	&&\vdots&&\\
	0 \ar[r]& H^0_{\m}\left(\frac{S}{(g_1,\ldots,g_{d})}\right) \ar[r]^-{\gamma_d}& H^1_{\m}\left(\frac{S}{(g_1,\ldots,g_{d-1})}\right)(-1,-\nu_d) \ar[r]^-{g_d}& H^1_{\m}\left(\frac{S}{(g_1,\ldots,g_{d-1})}\right),
}
\end{align}
\endgroup 
where $\gamma_i$'s are the connecting homomorphisms. Hence
\begin{equation}\label{eqq1}
0 \to H^0_{\m}(\sym(I)) \to H_{\m}^1(\sym(E_{d-1}))(-1,-\nu_d) \overset{g_d}{\lrt} H_{\m}^1(\sym(E_{d-1})),
\end{equation}
\begin{equation}\label{eqq2}
0 \to H_{\m}^1(\sym(E_{d-1})) \to H_{\m}^2(\sym(E_{d-2}))(-1,-\nu_{d-1}) \overset{g_{d-1}}{\lrt} H_{\m}^2(\sym(E_{d-2})).
\end{equation}

\begin{theorem}[with hypotheses as in \ref{ass}]\label{thm1}
We have an isomorphism of bi-graded $S$-modules \[\mathcal{K} \cong \Big\{ w \in H^2_{\m}\left(\frac{S}{(g_1,\ldots,g_{d-2})}\right)(-2, -\nu_d-\nu_{d-1}) \mid g_{d-1} \cdot w=0 \text{ and } g_d \cdot w=0\Big\}.\]
\end{theorem}

\begin{proof}
	The commutative diagram 
	\begin{align*}
	\xymatrix@R0.5cm{
		\frac{S}{(g_1,\ldots,g_{d-2})}(-2,-\nu_{d-1}-\nu_d)  \ar[r]^-{g_{d-1}} \ar[d]^{g_d} & \frac{S}{(g_1,\ldots,g_{d-2})}(-1, -\nu_d) \ar[d]^{g_d}\\
		\frac{S}{(g_1,\ldots,g_{d-2})}(-1,-\nu_{d-1})  \ar[r]^-{g_{d-1}}  & \frac{S}{(g_1,\ldots,g_{d-2})} }
	\end{align*}
	can be extended to the following commutative diagram with exact rows 
	
	\begin{align*}
	\xymatrix@C0.75cm@R0.5cm{
		0 \ar[r] &\frac{S}{(g_1,\ldots,g_{d-2})}(-2,-\nu_{d-1}-\nu_d)  \ar[r]^-{g_{d-1}} \ar[d]^{g_d} & \frac{S}{(g_1,\ldots,g_{d-2})}(-1, -\nu_d) \ar[d]^{g_d} \ar[r] & \sym(E_{d-1})(-1, -\nu_d) \ar[r] \ar[d]^{g_d}& 0 \\
		0 \ar[r] & 	\frac{S}{(g_1,\ldots,g_{d-2})}(-1,-\nu_{d-1})  \ar[r]^-{g_{d-1}}  & \frac{S}{(g_1,\ldots,g_{d-2})} \ar[r] & \sym(E_{d-1}) \ar[r]& 0}
	\end{align*}
	By \eqref{eqq2} we get the following commutative diagram of exact rows 
	{\small 
	\begin{align*}
	\xymatrix@C0.45cm@R0.5cm{
		0 \ar[r]& H^1_{\m}(\sym(E_{d-1}))(-1, -\nu_{d}) \ar[r]^-{\gamma} \ar[d]^{g_d}& H^2_{\m}\left(\frac{S}{(g_1,\ldots,g_{d-2})}\right)(-2, -\nu_{d-1}-\nu_{d}) \ar[r]^-{g_{d-1}} \ar[d]^{g_d}& H^2_{\m}\left(\frac{S}{(g_1,\ldots,g_{d-2})}\right)(-1, -\nu_{d}) \ar[d]^{g_d}\\
		0 \ar[r]& H^1_{\m}(\sym(E_{d-1})) \ar[r]^-{\gamma}& H^2_{\m}\left(\frac{S}{(g_1,\ldots,g_{d-2})}\right)(-1, -\nu_{d-1}) \ar[r]^-{g_{d-1}}& H^2_{\m}\left(\frac{S}{(g_1,\ldots,g_{d-2})}\right)}
	\end{align*}}
	From the above diagram and from \eqref{eqq1}, we get the exact sequence 
	\begin{align*}
	0 \to \mathcal{K} \to & \ker\left(H^2_{\m}(S/(g_1,\ldots,g_{d-2}))(-2, -\nu_{d-1}-\nu_{d}) \overset{g_d}{\longrightarrow} H^2_{\m}(S/(g_1,\ldots,g_{d-2}))(-1, -\nu_{d-1})\right) \\ &\overset{g_{d-1}}{\longrightarrow} \ker\left(H^2_{\m}(S/(g_1,\ldots,g_{d-2}))(-1, -\nu_{d}) \overset{g_d}{\longrightarrow} H^2_{\m}(S/(g_1,\ldots,g_{d-2})) \right).
	\end{align*}
	From this we get that 	
	\begin{align*}
		\mathcal{K} &\cong \{ w \in H^2_{\m}(S/(g_1,\ldots,g_{d-2}))(-2, -\nu_{d-1}-\nu_{d}) \mid g_{d-1} \cdot w=0 \text{ and } g_d \cdot w=0\}\\
	&\cong \{ w \in H^2_{\m}(S/(g_1,\ldots,g_{d-2}))(-2, -\nu+\sum_{i=1}^{d-2}\nu_i) \mid g_{d-1} \cdot w=0 \text{ and } g_d \cdot w=0\} \quad (\text{as } \sum_{i=1}^{d}\nu_i=\nu).
	\end{align*}
	\end{proof}

\s Recall that $H^d_{\m}(R)$ is an $R$ module via the action
\begin{equation*}
(x_1^{\alpha_1}\cdots x_d^{\alpha_d})\cdot (x_1^{-\beta_1}\cdots x_d^{-\beta_d})= \begin{cases}
x_1^{\alpha_1-\beta_1}\cdots x_d^{\alpha_d-\beta_d} & \text{if } \alpha_i< \beta_i \text{ for all }i \\
0  & \text{otherwise}.
\end{cases}
\end{equation*}

The following result gives us the highest possible $x$-degree for an element in the graded part $\mathcal{K}_{p,*}$.
\begin{corollary}[with hypotheses as in \ref{ass}]\label{vanishing-nonvanishing}
	The following statements hold:
	\begin{enumerate}[{\rm(i)}]
		\item For $p \ge d$ the graded part $\mathcal{K}_{p,*}$ is a finite dimensional $k$-vector space with $\mathcal{K}_{p,u}=0$ for $u > \nu-d$.
		\item $\mathcal{K}_{p,\nu-d} \neq 0$ and $\mathcal{K}_{*,\nu-d}\cong U(-d)$ is an isomorphism of graded $U$-modules.
	\end{enumerate}	
\end{corollary}
\begin{proof}
	(i) Fix $p \geq d$. Set $\mathfrak{d}(j)=\sum_{l=1}^{j}\nu_l$. If \[H^2_{\m}(S/(g_1,\ldots,g_{d-2}))(-2, -\nu+\mathfrak{d}(d-2))_{p,u}=H_{\m}^{2}(S/(g_1,\ldots,g_{d-2})_{p-2,u-\nu+\mathfrak{d}(d-2)}=0,\] then by Theorem \ref{thm1} it follows that $\mathcal{K}_{p,u}=0$. From \eqref{partial-g-rel} and \eqref{e0} we can say that $\gamma_i$ makes a shift of bi-degree $(-1,-\nu_i)$ for all $1 \leq i \leq d$ and so we get 
	{\footnotesize
	\begin{align}\label{eq-x-deg}
	\xymatrix@R0.5cm@C0.5cm{
		0\ar[r]& H_{\m}^{2}(S/(g_1,\ldots,g_{d-2}))_{p-2,u-\nu+\mathfrak{d}(d-2)} \ar@{^{(}->}[r]^-{\gamma_{d-2}}& \cdots \ar@{^{(}->}[r]^-{\gamma_3} & H_{\m}^{d-2}(S/(g_1,g_2))_{p-d+2,u-\nu+\mathfrak{d}(2)} \ar@{^{(}->}[d]^-{\gamma_2}&\\
		&&& H_{\m}^{d-1}(S/(g_1))_{p-d+1,u-\nu+\mathfrak{d}(1)} \ar@{^{(}->}[r]^-{\gamma_1}& H_{\m}^d(S)_{p-d,u-\nu}}.
		\end{align}}
	Note that $S= \bigoplus_{\beta}R {\bf T}^{\beta}$ and $H_{\m}^d(S)=\bigoplus_{\beta}H_{\m}^d(R) {\bf T}^{\beta}$. For any $u> \nu-d$ we have $u-\nu>-d$. So $H_{\m}^d(S)_{(p-d,u-\nu)}=0$ and hence $H_{\m}^{2}(S/(g_1,\ldots,g_{d-2}))_{p-2,u-\nu+\mathfrak{d}(d-2)}=0$. The result follows.
	
(ii) Let $u=\nu-d$. Note that $H_{\m}^d(S)_{p-d,-d}=\bigoplus_{|\beta|=p-d}k \cdot \frac{1}{x_1\cdots x_d} {\bf T}^{\beta}$. Recall that $g_1=\sum_{i=1}^{d+1} a_{i,1}T_i$, with $a_{i,1} \in \m$. Therefore $a_{i,1}=x_1^{\alpha_{i1}}\cdots x_d^{\alpha_{id}}$ with $\alpha_{ij}>0$ for some $j$ if $a_{i,1} \neq 0$. So $a_{i,1} \cdot \frac{1}{x_1\cdots x_d}=0$ for all $1 \leq i \leq {d+1}$ and hence $g_1 \cdot \frac{1}{x_1\cdots x_d}=0$. Similarly we get that $g_l \cdot \frac{1}{x_1\cdots x_d}=0$ for all $2 \leq l \leq d$ ($\star$). Thus from \eqref{e0} it follows that \[H^{d-1}_{\m}(S/(g_1))_{p-d+1, -d+\nu_1} \cong \ker \left( H_{\m}^d(S)_{p-d,-d} \overset{g_1}{\lrt} H_{\m}^d(S)_{p-d+1,-d+\nu_1}\right)= H_{\m}^d(S)_{p-d,-d}.\] Using this fact iteratively we get that 
\begin{align*}
&H^{d-i}_{\m}(S/(g_1, \ldots, g_i))_{p-d+i, -d+\mathfrak{d}(i)} \\
\cong& \ker\left(H_{\m}^{d-i+1}(S/(g_1, \ldots, g_{i-1}))_{p-d+i-1,-d+\mathfrak{d}(i-1)} \overset{g_i}{\lrt} H_{\m}^{d-i+1}(S/(g_1, \ldots, g_{i-1}))_{p-d+i, -d+\mathfrak{d}(i)}\right)\\
=& H_{\m}^{d-i+1}(S/(g_1, \ldots, g_{i-1}))_{p-d+i-1,-d+\mathfrak{d}(i-1)} \left(\cong H_{\m}^d(S)_{p-d,-d}\right)
\end{align*}
for all $1 \leq i \leq d-2$. In particular, $H_{\m}^{2}(S/(g_1,\ldots,g_{d-2})_{p-2,-d+\mathfrak{d}(d-2)} \cong H_{\m}^d(S)_{p-d,-d}$. Again by the similar argument we can say that 
\begin{align*}
\mathcal{K}_{p,\nu-d} &\cong \{w \in H^2_{\m}(S/(g_1,\ldots,g_{d-2}))_{p-2,-d+\mathfrak{d}(d-2)} \mid g_{d-1} \cdot w=0 \text{ and } g_d \cdot w=0\} \\
&\cong \{w \in H_{\m}^d(S)_{p-d,-d} \mid g_{d-1} \cdot w=0 \text{ and } g_d \cdot w=0\}.
\end{align*}
Using ($\star$) we get $\mathcal{K}_{p,\nu-d}\cong H_{\m}^d(S)_{p-d,-d} \neq 0$. Note that $H_{\m}^d(S)_{*,-d}= \bigoplus k \cdot \frac{1}{x_1\cdots x_d} {\bf T}^{\gamma} \cong U(-d)$ as graded $U$-modules. The result follows.
\end{proof}

\begin{remark}\label{first-differ}
Note that $H_{\m}^d(S)$ is $\Z^2$-graded with $H_{\m}^d(S)_{i,*}=0$ for all $i<0$. From \ref{eq-x-deg} we have \[0 \to H_{\m}^{2}(S/(g_1,\ldots,g_{d-2}))_{p-2,u-\nu+\mathfrak{d}(d-2)} \xrightarrow{\gamma_{1} \circ \cdots \circ \gamma_{d-2}} H^d_{\m}(S)_{p-d,u-\nu}.\] Thus $H_{\m}^{2}(S/(g_1,\ldots,g_{d-2}))_{p-2,u-\nu+\mathfrak{d}(d-2)}=0$ and hence $\mathcal{K}_{p,u}=0$ for all $u \geq 0$ when $p<d$. By part (ii) of Corollary \ref{vanishing-nonvanishing} it follows that $\sym(I)$ and $\R(I)$ first differ in degree $d$ for any $d \geq 3$.	
\end{remark}

\section{Translation into \texorpdfstring{$\D$}{D}-modules}
We define $\mathcal{T}$ as a polynomial ring in $d+1$ variables over the Weyl algebra $A_d(k)$, that is, $\mathcal{T}=k[x_1,\ldots,x_d]\langle\partial_1,\ldots,\partial_d\rangle[T_1,\ldots,T_{d+1}]$.

By $\mathcal{F}$ we denote the automorphism ({\it Fourier transform}) on $\mathcal{T}$ defined by 
\[\mathcal{F}(x_i)=\partial_i, \quad \mathcal{F}(\partial_i)=-x_i, \quad \mathcal{F}(T_i)=T_i.\]

Let $S_{\mathcal{F}}$ denote $S$ twisted by $\mathcal{F}$. Then we know that $S_{\mathcal{F}}$ is a $\mathcal{T}$-module, where $t \star b= \mathcal{F}(t) \bullet b$ for any $b \in S_{\mathcal{F}}$ and $t \in \mathcal{T}$. Note that $t \ast b= \mathcal{F}^{-1}(t) \star b= \mathcal{F}(\mathcal{F}^{-1}(t))\bullet b= t \bullet b$ for any $b \in (S_{\mathcal{F}})_{\mathcal{F}^{-1}}$ and $t \in \mathcal{T}$ and hence $(S_{\mathcal{F}})_{\mathcal{F}^{-1}}=S$.  Clearly $S_{\mathcal{F}}$ is also a $S$-module with the same operation (as $S \subseteq \mathcal{T}$ is a subring). For the rest of this section we use the notations $L_i=\mathcal{F}(g_i)$ for $1 \leq i \leq d$. For all $1 \leq i_1< \cdots< i_r \leq d$, define \[\sol(L_{i_1},\ldots,L_{i_r};S)=\{h \in S \mid L_j \bullet h=0 \text{ for all } j=i_1, \ldots, i_r\}.\] By the similar arguments as in the proof of \cite[Lemma 3.9]{YC}, we can show that the $k$-vector space $\sol(L_{i_1},\ldots,L_{i_r};S)$ has a $S$-module structure given by the twisting of the Fourier transform: if $f \in S$ and $h \in \sol(L_{i_1},\ldots,L_{i_r};S)$, then $f \diamond h=\mathcal{F}(f) \bullet h$. Therefore \[\sol(L_{i_1},\ldots,L_{i_r};S)=\{h \in S_{\mathcal{F}} \mid g_j \star h=0 \text{ for all } j=i_1, \ldots, i_r\}\] is a $S$-submodule of $S_{\mathcal{F}}$.

\begin{remark}
For any $\omega \in H^2_{\m}(S/(g_1,\ldots,g_{d-2}))$ we have $g_i \bullet \omega=g_i \cdot \omega$ and so we get that 
\begin{align*}
&\{\omega \in H^2_{\m}(S/(g_1,\ldots,g_{d-2})) \mid g_{d-1} \bullet \omega=0 \text{ and } g_{d} \bullet \omega=0\}\\
=&\{\omega \in H^2_{\m}(S/(g_1,\ldots,g_{d-2})) \mid g_{d-1} \cdot \omega=0 \text{ and } g_d \cdot \omega=0\}.
\end{align*}
\end{remark}

Since $S= \bigoplus_{\beta}R {\bf T}^{\beta}$ and $H_{\m}^d(S)=\bigoplus_{\beta}H_{\m}^d(R) {\bf T}^{\beta}$ so similar way as in \cite[Proposition 3.5]{YC} we can show that, the left $\mathcal{T}$-modules $H_{\m}^d(S)$ and $S$ are cyclic with generators $\frac{1}{x_1\cdots x_d}$ and $1$ respectively. They have the following presentations:
\begin{align}\label{ali-d0}
& 0 \to \mathcal{T}(x_1,\ldots,x_d) \to \mathcal{T} \overset{\bullet \frac{1}{x_1\cdots x_d}}{\lrt} H^d_{\m}(S) \to 0, \\
&  0 \to \mathcal{T}(\partial_1,\ldots,\partial_d) \to \mathcal{T} \overset{\bullet 1}{\lrt}S \to 0.
\end{align}
(as $x_i \cdot 1/x_1\cdots x_d=0$ and $\partial_i \cdot 1/\partial_1\cdots\partial_d=0$ for all $1 \leq i \leq d$). Hence $S \cong \mathcal{T}/\mathcal{T}(\partial_1,\ldots,\partial_d)$ and $H_{\m}^d(S) \cong \mathcal{T}/\mathcal{T}(x_1,\ldots,x_d)$. From the proof of \cite[Theorem 3.11]{YC} (Step-1), we have an induced isomorphism of left $\mathcal{T}$-modules \[\overline{\mathcal{F}}: H_{\m}^d(S) \to S_{\mathcal{F}}.\] The isomorphism satisfies $\overline{\mathcal{F}}(\Pi_{\partial}(z))= \Pi_{x}(\mathcal{F}(z))$ for all $z \in \mathcal{T}$, where
\begin{align*}
\Pi_{x}: \mathcal{T} \to \mathcal{T}/(\partial_1,\ldots,\partial_d) (\cong S) \quad \text{and} \quad \Pi_{\partial}: \mathcal{T} \to \mathcal{T}/(x_1,\ldots,x_d) (\cong H_{\m}^d(S)),
\end{align*}
that is, the following diagram 
\begin{align*}
\xymatrix@R0.5cm{
	\mathcal{T} \ar[r]^{\mathcal{F}} \ar[d]_{\Pi_{\partial}}& \mathcal{T} \ar[d]^{\Pi_x}\\
	\frac{\mathcal{T}}{(x_1,\ldots,x_d)} \ar[r]^{\overline{\mathcal{F}}} & \frac{\mathcal{T}}{(\partial_1,\ldots,\partial_d)}}
\end{align*}
commutes. Since $\mathcal{F}$ is an isomorphism and $z \in \mathcal{T}(x_1,\ldots,x_d)$ if and only if $\mathcal{F}(z) \in \mathcal{T}(\partial_1,\ldots,\partial_d)$ for any $z \in \mathcal{T}$ so $\overline{\mathcal{F}}$ is an isomorphism. Again  $\overline{\mathcal{F}}$ is $\mathcal{T}$ linear which follows from the following claim.

\begin{claim}\label{cl0}
	$\overline{\mathcal{F}}(t \bullet b)=\mathcal{F}(t) \bullet \overline{\mathcal{F}}(b)= t \star \overline{\mathcal{F}}(b)$ for any $t \in \mathcal{T}$ and $b \in \mathcal{T}/(x_1,\ldots,x_d) (\cong H_{\m}^d(S))$.	
\end{claim} 

Recall $\overline{\mathcal{F}}(\Pi_{\partial}(z))= \Pi_{x}(\mathcal{F}(z))$ for all $z \in \mathcal{T}$. Since $\Pi_{\partial}$ is surjective so there exists some $c \in \mathcal{T}$ such that $\Pi_{\partial}(c)=b$. So $\overline{\mathcal{F}}(t \bullet b)= \overline{\mathcal{F}}(\Pi_{\partial}(t) \bullet b)=\overline{\mathcal{F}}(\Pi_{\partial}(t) \bullet \Pi_{\partial}(c))= \overline{\mathcal{F}}\circ \Pi_{\partial}(t \bullet c)= \Pi_{x} \circ \mathcal{F}(t \bullet c)= \Pi_x (\mathcal{F}(t)) \bullet \Pi_x (\mathcal{F}(c))= \mathcal{F}(t) \bullet \overline{\mathcal{F}}\left(\Pi_{\partial}(c)\right)=\mathcal{F}(t) \bullet \overline{\mathcal{F}}(b)=t\star \overline{\mathcal{F}}(b)$. The claim follows.

Using \eqref{e0} we get the following diagram  
\begin{align*}
\xymatrix@R0.5cm@C0.5cm{
	0\ar[r]& H_{\m}^{2}(S/(g_1,\ldots,g_{d-2})) \ar@{^{(}->}[r]^-{\gamma_{d-2}}& \cdots \ar@{^{(}->}[r]^-{\gamma_3} & H_{\m}^{d-2}(S/(g_1,g_2)) \ar@{^{(}->}[r]^-{\gamma_2}& H_{\m}^{d-1}(S/(g_1)) \ar@{^{(}->}[r]^-{\gamma_1}& H_{\m}^d(S) \ar[r]^{\cdot g_1} \ar[d]_{\overline{\mathcal{F}}}& H_{\m}^d(S) \ar[d]^{\overline{\mathcal{F}}} \\
	&&&&& S_{\mathcal{F}}\ar[r]^{\cdot g_1} &S_{\mathcal{F}}}
\end{align*}
which induces an injective map $\theta:= \overline{\mathcal{F}} \circ\gamma_1 \circ \cdots \circ\gamma_{d-2}: H_{\m}^2(S/(g_1,\ldots,g_{d-2})) \to S_{\mathcal{F}}$. 

Following \cite[Notation 3.10]{YC} we shall write $\mathcal{S}= \sol(L_1,\ldots,L_d;S)_{\mathcal{F}}$ to stress the bigraded $S$-module structure induced on $\sol(L_1,\ldots,L_d;S)$ by the twisting of the Fourier transform ${\mathcal{F}}$.

\begin{theorem}[with hypotheses as in \ref{ass}]\label{MT2}
	We have the following isomorphism of bigraded $S$-modules \[\mathcal{K} \cong \mathcal{S}(-d, -\nu+d),\] where $\mathcal{S}=\sol(L_1,\ldots,L_d;S)_{\mathcal{F}}$.
\end{theorem}

\begin{proof}
	(Step-1:) Set $\gamma:=\gamma_1 \circ \cdots \circ\gamma_{d-2}$. As each $\gamma_i$ is an injective map so is $\gamma$. For any $\omega \in H^2_{\m}(S/(g_1,\ldots,g_{d-2}))$ we have $\gamma(w) \in H^d_{\m}(S)$ and hence $\gamma(w)= \Pi_{\partial}(z)$ for some $z \in \mathcal{T}$ (as $\Pi_{\partial}$ is a surjective map). Now $g_i \bullet \Pi_{\partial}(z)=g_i \cdot \Pi_{\partial}(z)=g_i \cdot \gamma(w)= \gamma (g_i \cdot w)$ (as $\gamma$ is a $S$-linear map). Since $\gamma$ is an injective map so we get that $g_i \bullet \Pi_{\partial}(z)=0$ if and only if $g_i \cdot w=g_i \bullet w=0$. Thus for any $z \in \mathcal{T}$ we get the following equivalences 
	
	\begingroup
	\allowdisplaybreaks
	\begin{align*}
	\left( {\begin{array}{c}
		g_{d-1} \bullet \Pi_{\partial}(z)=0 \\
		g_d \bullet \Pi_{\partial}(z)=0 \\
		\end{array} } \right)
	&\iff
	\left( {\begin{array}{c}
		g_{d-1}z \in \mathcal{T}(x_1,\ldots,x_d)\\
		g_dz \in \mathcal{T}(x_1,\ldots,x_d) \\
		\end{array} } \right) 
	\iff  
	\left( {\begin{array}{c}
		\mathcal{F}(g_{d-1})\mathcal{F}(z) \in \mathcal{T}(\partial_1,\ldots,\partial_d)\\
		\mathcal{F}(g_d)\mathcal{F}(z) \in \mathcal{T}(\partial_1,\ldots,\partial_d) \\
		\end{array} } \right) \\
	&\iff
	\left( {\begin{array}{c}
		L_{d-1} \bullet \Pi_{x}({\mathcal{F}(z)})=0 \\
		L_d \bullet \Pi_{x}({\mathcal{F}(z)})=0 \\
		\end{array} } \right) 
	\iff  
	\left( {\begin{array}{c}
		L_{d-1} \bullet \overline{\mathcal{F}}(\Pi_{\partial}(z))=0 \\
		L_d \bullet \overline{\mathcal{F}}(\Pi_{\partial}(z))=0 \\
		\end{array} } \right). 
	\end{align*}
	\endgroup
		Recall that $\theta:=\overline{\mathcal{F}} \circ \gamma:H^2_{\m}(S/(g_1, \ldots, g_{d-2})) \to S_{\mathcal{F}}$ is an injective map. Since $\gamma$ is $S$-linear and $\overline{\mathcal{F}}$ is $\mathcal{T}$-linear so $\theta$ is $S$-linear and hence $\im \theta$ is a $S$-submodule of $S_{\mathcal{F}}$.
		Due to the above observation $\theta$ induces an isomorphism of $S$-modules 
	
		\begin{equation}\label{eqqq1}
		\begin{split}
	&\{\omega \in H^2_{\m}(S/(g_1,\ldots,g_{d-2})) \mid g_{d-1} \bullet \omega=0 \text{ and } g_d \bullet \omega=0\}\\
	\cong_S& \{b \in \im \theta \mid g_{d-1} \star b=0, g_d \star b=0\}.
	   \end{split}
		\end{equation}
		(as $b=\theta(\omega)$ for some $\omega \in H^2_{\m}(S/(g_1))$). We claim that $\im \theta =(0:_{S_{\mathcal{F}}} (g_1, \ldots, g_{d-2}))$. Note that $g_i \cdot \gamma(\omega)=\gamma_1(\cdots\gamma_{i-1}(g_i \cdot \gamma_i(\omega_i))\cdots)=0$ for all $2 \leq i \leq d-2$ and $g_1 \cdot \gamma_1(w_1)=0$ where $w_i=\gamma_{i+1}\circ \cdots \circ \gamma_{d-2}(\omega)$ for all $1 \leq i \leq d-2$ (as $\gamma_i$ is $S$-linear, $g_i \in S$ and $g_i \cdot \gamma_i=0$). Thus for any $\omega \in H^2_{\m}(S/(g_1, \ldots, g_{d-2}))$ we get that $g_i \star \theta(\omega)=\overline{\mathcal{F}}(g_i \cdot \gamma(\omega))=0$ for all $1 \leq i \leq d-2$ and hence $(g_1, \ldots, g_{d-2}) \subseteq (0:_S H^2_{\m}(S/(g_1, \ldots, g_{d-2})))$. So $(0:_{S_{\mathcal{F}}} \im \theta) \supseteq \mathcal{F}((g_1, \ldots, g_{d-2}
		))=(\mathcal{F}(g_1), \ldots, \mathcal{F}(g_{d-2}))=(L_1, \ldots, L_{d-2})$. Thus $\im \theta \subseteq (0:_{S_{\mathcal{F}}} (g_1, \ldots, g_{d-2}))$. On the other hand, for any $y \in (0:_{{S}_{\mathcal{F}}} (g_1, \ldots, g_{d-2})) \subseteq S_{\mathcal{F}}$, there exists some $a \in H^d_{\m}(S)$ such that $\overline{\mathcal{F}}(a)=y$ (as $\overline{\mathcal{F}}$ is onto). Note that $0=g_1 \star y=\mathcal{F}(g_1)\bullet \overline{\mathcal{F}}(a)= \overline{\mathcal{F}}(g_1 \cdot a)$ (by Claim \ref{cl0}) and hence $g_1 \cdot a=0$ (as $\overline{\mathcal{F}}$ is one-one). So $a \in \ker \left(H_{\m}^d(S) \overset{g_1}{\lrt} H_{\m}^d(S) \right)= \im \gamma_1= \gamma_1(H^{d-1}_{\m}(S/(g_1)))$. Let $a=\gamma_1(a_1)$ for some $a_1 \in H^{d-1}_{\m}(S/(g_1))$. Now $0=g_2 \star y=\mathcal{F}(g_2) \bullet \overline{\mathcal{F}}(a)=\mathcal{F}(g_2) \bullet \overline{\mathcal{F}}(\gamma_1(a_1))=\overline{\mathcal{F}}(g_2 \cdot\gamma_1(a_1))=\overline{\mathcal{F}} \circ \gamma_1(g_2 \cdot a_1)$ (as $\gamma_1$ is $S$-linear). Since both $\overline{\mathcal{F}}$ and $\gamma_1$ are one-one so we have $g_2 \cdot a_1=0$. Thus $a_1 \in \ker \left(H_{\m}^{d-1}(S/(g_1)) \overset{g_2}{\lrt} H_{\m}^{d-1}(S/(g_1)) \right)= \im \gamma_2= \gamma_2(H^{d-2}_{\m}(S/(g_1,g_2)))$. Therefore $a \in \im \gamma_1 \circ \gamma_2$. Proceeding in this way we get that $a \in \im \gamma$ and hence $y \in \im \theta$. Thus $(0:_{{S}_{\mathcal{F}}}(g_1,\ldots,g_{d-2})) \subseteq \im \theta$. The claim follows. 

		Hence from \eqref{eqqq1} it follows that 
		\begin{align*}
		& \{\omega \in H^2_{\m}(S/(g_1,\ldots,g_{d-2})) \mid g_{d-1} \bullet \omega=0 \text{ and } g_d \bullet \omega=0\}\\
		& \cong_S \{b \in S_{\mathcal{F}} \mid g_i \star b=0 \text{ for all } i=1, \ldots, d\}\\
		& = \{b \in S \mid L_i \bullet b=0 \text{ for all } i=1, \ldots, d\}_{\mathcal{F}}=\sol(L_1,\ldots,L_d;S)_{\mathcal{F}}.
		\end{align*}

	\noindent
	(Step-2:) From the definition of $\mathcal{F}$ we have that $\overline{\mathcal{F}}$ is homogeneous of degree $0$ on $T_i$'s. Since \[\partial^{\alpha_1-1}\cdots\partial_d^{\alpha_d-1} \bullet \frac{1}{x_1\cdots x_d}=(-1)^{\alpha_1+\cdots+\alpha_d} \frac{(\alpha_1-1)!\cdots(\alpha_d-1)!}{x_1^{\alpha_1} \cdots x_d^{\alpha_d}}\] so we have
	\[\frac{1}{x_1^{\alpha_1} \cdots x_d^{\alpha_d}}=(-1)^{\alpha_1+\cdots+\alpha_d}\frac{\partial_1^{\alpha_1-1}\cdots\partial_d^{\alpha_d-1}}{(\alpha_1-1)!\cdots(\alpha_d-1)!} \bullet \frac{1}{x_1\cdots x_d}\in H_{\m}^d(R).\] and makes a shift degree of $d$ in the $x_i$'s. By \eqref{ali-d0} it follows that the map $\Pi_{\partial}: \mathcal{T} \to \mathcal{T}/(x_1,\ldots,x_d)$ is defined by $z \mapsto z \bullet \frac{1}{x_1\cdots x_d}$. Thus we get that 
	\begingroup
	\allowdisplaybreaks
	\begin{align*}
	\overline{\mathcal{F}}\left(\frac{1}{x_1^{\alpha_1} \cdots x_d^{\alpha_d}}\right)&= \overline{\mathcal{F}}\left(\Pi_{\partial}\left((-1)^{\alpha_1+\cdots+\alpha_d}\frac{\partial_1^{\alpha_1-1}\cdots\partial_d^{\alpha_d-1}}{(\alpha_1-1)!\cdots(\alpha_d-1)!}\right)\right)\\
	&=\mathcal{F}\left((-1)^{\alpha_1+\cdots+\alpha_d}\frac{\partial^{\alpha_1-1}\cdots\partial_d^{\alpha_d-1}}{(\alpha_1-1)!\cdots(\alpha_d-1)!}\right)\\
	&=(-1)^{\alpha_1+\cdots+\alpha_d}\frac{x_1^{\alpha_1-1}\cdots x_d^{\alpha_d-1}}{(\alpha_1-1)!\cdots(\alpha_d-1)!} \in R.
	\end{align*}
	\endgroup
	Hence we have that $\mathcal{F}$ makes a shift degree of $d$ in $x_i$'s. Moreover, from \eqref{e0} we can say that $\gamma_i$ makes a shift of degree $-1$ in $T_j$'s and of degree $-\nu_i$ in $x_j$'s for all $i=1, \ldots, d$. Therefore $\theta=\overline{\mathcal{F}} \circ \gamma$ makes a shift of degree $-\sum_{i=1}^{d-2}\nu_i+d$ in $x_j$'s and of degree $-(d-2)=-d+2$ in $T_j$'s. Adding the shift degree $(-d+2,-\sum_{i=1}^{d-2}\nu_i+d)$ to Theorem \ref{thm1} we get the result. 
\end{proof}

Cid proved \cite[Proposition 3.13]{YC} observing that the way used in \cite[Chapter 6, Theorem 1.2]{SCC} can be used to prove his result. He also used \cite[Subsection 3.1]{YC}. Both of them are independent of the number of variables and number of equations in a system of differential equations. So in the same way we get the following result.
\begin{proposition}[with hypotheses as in \ref{ass}]\label{hom}
We have the following isomorphism of graded $U$-modules \[\mathcal{K} \cong {}^*\Hom_{\mathcal{T}}(\mathcal{T}/\mathcal{T}(L_1,\ldots,L_d), S)(-d).\]
\end{proposition}

\section{The bigraded structure of \texorpdfstring{$\mathcal{K}$}{K} and its relation with \texorpdfstring{$b$}{b} functions}

\begin{notation}\label{not}
Fix the integers $p \geq d, m=\binom{p}{d}$ and $n=\binom{p+1}{d}$, where $\dim R=d \geq 3$. Note that if $p=d$, then $m=1, n=d+1$, i.e., $m=\#\{{\bf T}^{\beta} \mid |\beta|=0\}$ and $n=\#\{{\bf T}^{\beta} \mid |\beta|=1\}$; if $p=d+1$, then $m=d+1, n=\binom{(d+1)+1}{d}=\binom{(d+1)+(2-1)}{(d+1)-1}$, i.e., $m=\#\{{\bf T}^{\beta} \mid |\beta|=1\}$ and $n=\#\{{\bf T}^{\beta} \mid |\beta|=2\}$.

The graded part $\mathcal{S}_{p-d,*}$ is given as the solution set of the system of differential equations
\[V=\{h =(h_1,\ldots,h_m) \in R^m \mid [L_i] \bullet h=0 \text{ for all }i=1, \ldots d\},\] where $[L_i] \in D^{n \times m}$ is an $n \times m$ matrix with entries in $D$ and induced by restricting $L_i$ to the monomials ${\bf T}^{\beta}$ of degree $|\beta|=p-d$. We define a new matrix $H \in D^{dn \times m}$ defined by 
{\small
\[
H=\left(
\begin{array}{c}
L_1\\
\hline
L_2\\
\hline
\vdots\\
\hline
L_d
\end{array}\right)
.\]}
Then we can write $V=\{h =(h_1,\ldots,h_m) \in R^m \mid H \bullet h=0\}$.Set $N=D^{dn} \cdot H$. Clearly $N \subset D^m$ is a left $D$-module and image of the map $\phi: D^{dn} \to D^m$ induced by $H$. Set $M=D^m/N$. Following the same way as in \cite[Theorem 1.2]{SCC} we can show that $V \cong \Hom_{D}(M,R^m)$ as $k$ vector spaces. 
\end{notation}
\begin{example}\label{eg-not}
For $d=3$ we have $m=\binom{p}{3}$ and $n=\binom{p+1}{3}$. In this case we compute the system of differential equations when $p=2$ and $p=3$. Suppose 
\begin{align*}
L_1&= a_1T_1+a_2+T_2+a_3T_3+a_4T_4,\\
L_2&= b_1T_1+b_2+T_2+b_3T_3+b_4T_4,\\
L_3&= c_1T_1+c_2+T_2+c_3T_3+c_4T_4.
\end{align*}
For $p=3$, we have $h=(h_1) \in S_0=R$, and the equations $L_1 \bullet h=0, L_2 \bullet h=0$ and $L_3 \bullet h=0$ can be expressed as 
{\small
\begin{align*}
&\left(
\begin{array}{c}
a_1\\
a_2\\
a_3\\
a_4
\end{array}\right) \bullet (h_1)=0, \quad
\left(
\begin{array}{c}
b_1\\
b_2\\
b_3\\
b_4
\end{array}\right) \bullet (h_1)=0 \quad \text{ and } \quad 
\left(
\begin{array}{c}
c_1\\
c_2\\
c_3\\
c_4
\end{array}\right) \bullet (h_1)=0.
\end{align*}}
So in this case, $N=D(a_1,a_2,a_3,a_4,b_1,b_2,b_3,b_4,c_1,c_2,c_3,c_4)$. 

For $p=4$, we have $h=(h_1,h_2,h_3,h_4) \in S_1=R T_1+RT_2+RT_3+RT_4$. Note that $f=h_1T_2+h_2T_2+h_3T_3+h_4T_4 \in \mathcal{S}_{1,*}$ if $L_i \bullet f=0$ for all $i=1,2,3$. Now $L_1 \bullet f=0$ implies that $\left(\sum_{i=1}^4 a_i T_i\right) \bullet \left(\sum_{j=1}^4 h_j T_j\right)=0$, that is, $(a_1h_1) T_1^2+(a_2h_1+a_1h_2)T_1T_2+(a_3h_1+a_1h_3)T_1T_3+ \cdots + (a_4h_4) T_4^2=0$. Hence $a_1h_1=0, ~a_2h_1+a_1h_2=0, ~a_3h_1+a_1h_3=0, \ldots, a_4h_4=0$. Thus shorting the monomials ${\bf T}^{\beta}$ in lexicographical order, the equations $L_1 \bullet h=0$
can be expressed as 
\end{example}
{\small
\[
\begin{blockarray}{ccccc}	
\begin{block}{ccccc}
& T_1 & T_2 & T_3&T_4\\	
\end{block}
\begin{block}{c(cccc)}
T_1^2 & a_1 & 0 & 0 &0\\
T_1T_2 & a_2 & a_1 & 0 & 0\\			
T_1T_3 & a_3 & 0 & a_1 & 0\\	
T_1T_4 & a_4 & 0 & 0 & a_1\\		
T_2^2 & 0 & a_2 & 0 & 0\\			
T_2T_3 & 0 & a_3 & a_2 & 0\\	
T_2T_4 & 0 & a_4 & 0 & a_2\\		
T_3^2 & 0 & 0 & a_3& 0\\
T_3T_4 & 0 & 0 & a_4 & a_3\\
T_4^2 & 0 & 0 & 0& a_4\\
\end{block}			
\end{blockarray}
\bullet
\left(
\begin{array}{c}
h_1\\
h_2\\
h_3\\
h_4
\end{array}
\right)
= 0\]}
Replacing $a_i$ by $b_i$ (resp. $c_i$) we get the expression for the equation $L_2 \bullet h=0$ (resp. $L_3 \bullet h=0$).

\s Let $A \in D^{r\times s}$ be an $r\times s$ matrix with entries in $D$.
Multiplying with $A$ gives us a map $D^r \xrightarrow{\cdot A} D^s$ of left $D$-modules. Applying $\Hom_D(-,D)$ to this map induces the map ${(D^s)}^T \xrightarrow{A\cdot} {(D^r)}^T$ of right $D$-modules.

We have an equivalence between the category of left $D$-modules and the category of right $D$-modules, given by the algebra involution $\tau: D \to D$ defined by $f\partial^{\alpha} \mapsto (-1)^{|\alpha|}\partial^{\alpha} f$, where $f \in R$. The map $\tau$ is called the {\it standard transposition}.

Given a left $D$-module $D^r/M_0$, it can be checked that
\[\tau\Big(\frac{D^r}{M_0}\Big) = \frac{D^r}{\tau(M_0)},\] where $\tau(M_0)=\{\tau(L) \mid L \in M_0\}$.

\begin{proposition}\label{M-holo}
The left $D$-module $M=D^m/N$ is holonomic. 
\end{proposition}

\begin{proof}
Consider $\mathcal{T}$ as standard graded with $\deg r=0$ for all $r \in D$ and $\deg T_i=1$ for all $i$. From the graded part $p$ of \eqref{koszul} we get the following exact sequence of the left $D$-modules, 
\begin{equation} \label{hope}
0 \to \mathcal{T}_{p-d} \overset{\cdot A}{\lrt}\mathcal{T}_{p-d+1}^d \to \cdots \to \mathcal{T}_{p-1}^d \to \mathcal{T}_{p} \to Q_p \to 0,
\end{equation}
where $A=\left( (-1)^{d-1}[L_d]^T\mid \cdots\mid -[L_2]^T \mid [L_1]^T \right)$ and $[L_i]^T$ denotes the transpose of the matrix $[L_i]$ defined in Notation \ref{not}.
Applying $\Hom_D(-,D)$ to \eqref{hope} we get the following complex of right $D$-modules

\[0 \to (\mathcal{T}_{p})^T \to (\mathcal{T}_{p-1}^d)^T \to \cdots \to (\mathcal{T}_{p-d+1}^d)^T \overset{A \cdot}{\lrt} (\mathcal{T}_{p-d})^T \to  0,\] 
where the cokernel of the last map $A \cdot$ is $\Ext^d_D(Q_p,D)$, that is, $\Ext^d_D(Q_p,D) \cong \tau(M)$ (see Example \ref{eg-clr} for clarification). By \cite[Lemma 7.3]{Bjork} we have $d(\Ext^d_D(Q_p,D)) \leq 2d-d=d$. Since for any finitely generated $D$-module $N$, $d \leq d(N) \leq 2d$ so we get that $d(\Ext^d_D(Q_p,D))=d$ and hence it is a holonomic right $D$-module. The result follows.
\end{proof}

\begin{example}\label{eg-clr}
Let $d=3, p=3$ and $m,n,L_1,L_2,L_3$ as in Example \ref{eg-not}. Then 
\begin{align*}
A&=\left( [L_3]^T\mid -[L_2]^T \mid [L_1]^T \right)\\
&=\left(\begin{array}{cccccccccccccc}
c_1& c_2 &c_3&c_4 &\vline &-b_1& -b_2 &-b_3 &-b_4& \vline &a_1& a_2 &a_3&a_4 \end{array}\right) \in D^{1 \times 12}.
\end{align*} 
Note that $\mathcal{T}_{q-3} \cong D$ and $\mathcal{T}_{q-2}=DT_1+DT_2+DT_3+DT_4 \cong D^4$. Hence $(\mathcal{T}_{p-2}^3)^T \overset{A \cdot}{\lrt} (\mathcal{T}_{p-3})^T$ induces a map $((D^4)^3)^T \overset{A \cdot}{\lrt} (D)^T$. Now $\tau(M)= \tau(D^{12}/D \cdot A)= D^{12}/\tau(D \cdot A)$. So $\Ext^3_D(Q_p,D) \cong \tau(M)$. 
\end{example}

\s Let $\ell=\sum_{\alpha,\beta}C_{\alpha,\beta} x^{\alpha} \partial^\beta$ be an element in $D=k[x_1,\ldots,x_d]\langle\partial_1,\ldots,\partial_d\rangle$, where $x^{\alpha}=x_1^{\alpha_1}\cdots x_d^{\alpha_d}$ and $\partial^{\beta}=\partial_1^{\beta_1}\cdots\partial_d^{\beta_d}$. Let $w=(w_1,\ldots,w_d)$ be a given generic weight vector. Then the initial form of $l$ with respect to $w$ is denoted by $\inn_{(-w,w)}(\ell)$ and defined as \[\inn_{(-w,w)}(\ell)= \sum_{-\alpha \cdot w+\beta \cdot w \text{ is maximum}} C_{\alpha,\beta} x^{\alpha} \partial^\beta.\] To make $\deg(x_i)=1$ and $\deg(\partial_i)=-1$ we take the weight vector $(-1,\ldots,-1) \in \Z^d$. 

\begin{definition}\cite[Definition 4.4]{YC}
	Let $J \subset D$ be a left ideal, then the $k$-vector space
	$$
	\inn(J)=k\cdot\Big\{\inn(\ell) \mid \ell \in J \Big\}
	$$
	is a left ideal in $D$ and it is called the {\it initial ideal} of $J$.
\end{definition}

\begin{definition}\cite[Definition 4.5]{YC}
	Let $J \subset D$ be a holonomic left ideal. The elimination ideal 
	\begin{equation*}
	\inn(J) \;\cap\; k[-\sum_{i=1}^d x_i\partial_i]
	\end{equation*}
	is principal in the univariate polynomial ring $k[s]$, where $s=-\sum_{i=1}^d x_i\partial_i$ and the generator $b_J(s)$ of this ideal is called the {\it$b$-function} of $J$.
\end{definition}
In \cite{YC}, Cid gave the following definition for the $b$-function of a left $D$-module.
\begin{definition}\cite[Definition 4.6]{YC}
	\label{bF_module}
	Let $M^{'}$ be a holonomic left $D$-module given as the quotient module $M^{'}=D^r/N^{'}$.
	For each $i = 1,\ldots,r$ with the canonical projection $\pi_i:D^r\rightarrow D$ of $D^r$ onto the $i$-th component $e_i$, we define a left $D$-ideal 
	$$
	J_i=\pi_i(N^{'} \;\cap \;D\cdot e_i)=\big\{\ell \in D \mid (0,\ldots,\underbrace{\ell}_{i\text{-th}},\ldots,0) \in N^{'}  \big\}.
	$$
	Then the $b$-function of $M^{'}$ is given as the least common multiple of the $b$-functions of the $D$-ideals $J_i$, that is,
	$$
	b_{M^{'}}(s) ={\normalfont\text{LCM}}_{i=1,\ldots,r} \;\big(b_{J_i}(s)\big).
	$$ 
\end{definition}
For each $i=1,\ldots,r$, the canonical injection $D/J_i \hookrightarrow D^r/N^{'}$ implies that each ideal $J_i$ is holonomic. Thus $b_{J_i}(s)$ is a non-zero polynomial by \cite[Theorem 5.1.2]{SST} and hence the $b$-function of a holonomic module is a non-zero polynomial.

\begin{theorem}\label{root}
Consider the $b$-function $b_M(s)$ of the holonomic $D$-module $M$ defined in Notation \ref{not}. For any integer $u$, if $b_M(-\nu+d+u) \neq 0$ then we have $\mathcal{K}_{p,u}=0$.
\end{theorem}

\begin{proof}
We prove by contradiction. Suppose $\mathcal{K}_{p,u} \neq 0$. Then by Theorem \ref{MT2} there exists $0 \neq h \in \mathcal{S}_{p-d,-k}$, where $-k=-\nu+d+u$. Following Notation \ref{not} we can write $h=(h_1,\ldots,h_m) \in V$ with $\deg h_i=k$. Let $b_{J_i}(s)$ be the $b$-function corresponding to the left $D$-ideal \[J_i=\pi_i(N \cap D \cdot e_i)= \{\ell \in D \mid (0,\ldots,\underbrace{\ell}_{i\text{-th}},\ldots,0)\in N\}.\] Thus $b_{J_i}(s) \cdot e_i \bullet h=0$, which implies that $b_{J_i}(s) \bullet h_i=0$ and hence $b_{J_i}(-k)h_i=0$. As $b_M(-k) \neq 0$ so $b_{J_i}(-k) \neq 0$ (as $b_M(s)=\lcm_{i=1,\ldots,m}(b_{J_i}(s))$ by the result generalizing \cite[Definition 4.6]{YC} in our case). Therefore $h_i=0$ for all $i$, a contradiction. 
\end{proof}

\begin{corollary}[with hypotheses as in \ref{ass}]\label{divide_b_function}
Let $u$ be the lowest possible $x$-degree for an element in the graded part $\mathcal{K}_{p,*}$, that is, $\mathcal{K}_{p,u} \neq 0$ and $\mathcal{K}_{p,u-1}=0$. Then the polynomial $s(s+1)\cdots(s+\nu-d-u)$ divides the $b$-function $b_M(s)$.
\end{corollary}

\begin{proof}
Since $\mathcal{K}_{p,v} \neq 0 $ for all $v \geq u$ so by Theorem \ref{root} it follows that $b_M(-\nu+d+v) = 0$ for all $v \geq u$. Hence $s(s+1) \cdots s(\nu-d-u)$ divides $b_M(s)$.
\end{proof}

We have the following result from \cite[Lemma 4.10]{YC}.
\begin{lemma}\label{belong}
	For any $k \ge 0$ we have the identity 
	\begin{equation}\label{ideal-eq1}
		s(s+1)\cdots(s+k)={(-1)}^{k+1}\sum_{|\alpha|=k+1} \frac{(k+1)!}{\alpha_1!~\alpha_2! \cdots \alpha_d!} x^{\alpha}\partial^{\alpha}.
	\end{equation}
	Thus, we have that 
	\begin{enumerate}[(i)]
		\item $s(s+1)\cdots(s+k) \in D(\partial_1, \ldots, \partial_d)^{k+1}$, where $D(\partial_1,\ldots,\partial_d)^{k+1}$ denotes the left $D$-ideal generated by the elements $\{\partial_1^{\zeta_1}\cdots\partial_d^{\zeta_d}\mid \zeta_1+\cdots+\zeta_d = k+1\}$;
		\item $s(s+1)\cdots(s+k)$ is homogeneous, that is, $$
		\inn\big(s(s+1)\cdots(s+k)\big)=s(s+1)\cdots(s+k).$$
	\end{enumerate}
\end{lemma}

\begin{proof}
From the proof of \cite[Lemma 4.10]{YC} we have  \[x_i^{\beta_i}\partial_i^{\beta_i}(x_i\partial_i-\beta_i)=x_i^{\beta_i+1}\partial_i^{\beta_i+1}\] for all $\beta_i \geq 0$ and $1 \leq i,j \leq d$. For all $\alpha=(\alpha_1, \ldots, \alpha_d) \in \N^d$ we have
\begin{align*}
x^{\alpha+e_i} \partial^{\alpha+e_i}&=x_1^{\alpha_1} \cdots x_{i-1}^{\alpha_{i-1}} x_i^{\alpha_i+1}x_{i+1}^{\alpha_{i+1}} \cdots x_d^{\alpha_d}\partial_1^{\alpha_1} \cdots \partial_{i-1}^{\alpha_{i-1}} \partial_i^{\alpha_i+1}\partial_{i+1}^{\alpha_{i+1}} \cdots \partial_d^{\alpha_d}\\
&=x_1^{\alpha_1} \cdots x_{i-1}^{\alpha_{i-1}} x_{i+1}^{\alpha_{i+1}} \cdots x_d^{\alpha_d} \partial_1^{\alpha_1} \cdots \partial_{i-1}^{\alpha_{i-1}} \partial_{i+1}^{\alpha_{i+1}} \cdots \partial_d^{\alpha_d}(x_i^{\alpha_i+1}\partial_i^{\alpha_i+1})\\
&=x_1^{\alpha_1} \cdots x_{i-1}^{\alpha_{i-1}} x_{i+1}^{\alpha_{i+1}} \cdots x_d^{\alpha_d} \partial_1^{\alpha_1} \cdots \partial_{i-1}^{\alpha_{i-1}} \partial_{i+1}^{\alpha_{i+1}} \cdots \partial_d^{\alpha_d}\{x_i^{\alpha_i}\partial_i^{\alpha_i}(x_i\partial_i-\alpha_i)\}\\
&=x^{\alpha} \partial^{\alpha}(x_i\partial_i-\alpha_i)
\end{align*}
(as $x_i, \partial_j$ commute with each other when $i \neq j$).
Thus we get
\begin{align*}
(-1)\sum_{i=1}^d x^{\alpha+e_i} \partial^{\alpha+e_i}&=(-1)\sum_{i=1}^d x^{\alpha} \partial^{\alpha}(x_i\partial_i-\alpha_i)=x^{\alpha} \partial^{\alpha}\sum_{i=1}^d(-x_i\partial_i+\alpha_i)\\
&=x^{\alpha} \partial^{\alpha}(-\sum_{i=1}^dx_i\partial_i+\sum_{i=1}^d\alpha_i)=x^{\alpha} \partial^{\alpha}(s+k+1),
\end{align*}
where $|\alpha|=\alpha_1+\cdots+\alpha_d=k+1$.

To prove the result we use induction hypothesis on $k$. For $k=0$, the result is trivially true (take $\alpha= e_1, \ldots, e_d$). Now we show that the result is true for $k+1$ if it holds true for $k$. Then
\begin{equation}\label{induc-eq}
\begin{split}
s(s+1)\cdots(s+k)(s+k+1)&={(-1)}^{k+1}\sum_{|\alpha|=k+1} \frac{(k+1)!}{\alpha_1!~\alpha_2! \cdots \alpha_d!} x^{\alpha}\partial^{\alpha}(s+k+1)\\
&={(-1)}^{k+2}\sum_{|\alpha|=k+1} \frac{(k+1)!}{\alpha_1!~\alpha_2! \cdots \alpha_d!} \sum_{i=1}^d x^{\alpha+e_i}\partial^{\alpha+e_i}.
\end{split}
\end{equation}
Note that for any $\beta=(\beta_1, \ldots, \beta_d) \in \N^d$ with $|\beta|=k+2$, we can write $\beta=\alpha+e_i$ for some $i$ and $\alpha=(\alpha_1, \ldots, \alpha_d) \in \N^d$ with 
and $|\alpha|=k+1$. Again $\alpha+e_i=\alpha'+e_j$, where $\alpha'=(\alpha_1, \ldots, \alpha_j-1, \ldots, \alpha_i+1, \ldots, \alpha_d)$ for all $j \neq i$ (position of $i,j$ may vary). Choose any $i \in \{1, \ldots, d\}$ and fix it. Then for any $i \neq j$ we claim that \[\frac{(k+1)!}{\alpha_1! \cdots \alpha_d!}+\sum_{\substack{j=1\\j \neq i}}^d \frac{(k+1)!}{\alpha_1! \cdots (\alpha_j-1)! \cdots (\alpha_i+1)! \cdots \alpha_d!}=\frac{(k+2)!}{\alpha_1!  \cdots (\alpha_i+1)! \cdots \alpha_d!}.\]
We prove the sub-claim by induction on $d$. We know $\binom{k+1}{j}+\binom{k+1}{j-1}=\binom{k+2}{j}$, that is, \[\frac{(k+1)!}{j!~  i!}+\frac{(k+1)!}{(j-1)!~(i+1)!}=\frac{(k+2)!}{j!(i+1)!},\]where $i+j=k$. This takes care of the base case. Let the sub-claim is true for any $r<d$. Let us assume that $i \neq 1$. Since $k+1-\alpha_1 \leq k+1$, so we have $K_1:=(k+1)!/(k+1-\alpha_1)! \in \N$. Now

\begingroup
\allowdisplaybreaks
\begin{align*}
&\frac{(k+1)!}{\alpha_1! \cdots \alpha_d!}+\sum_{\substack{j=1\\j \neq i}}^d \frac{(k+1)!}{\alpha_1! \cdots (\alpha_j-1)! \cdots (\alpha_i+1)! \cdots \alpha_d!} \\
=&\frac{K_1}{\alpha_1!}\left(\frac{(k+1-\alpha_1)!}{\alpha_2! \cdots  \alpha_d!}+\sum_{\substack{j=2\\j \neq i}}^d \frac{(k+1-\alpha_1)!}{\alpha_2! \cdots (\alpha_j-1)! \cdots (\alpha_i+1)! \cdots \alpha_d!}\right)+ \frac{(k+1)!}{(\alpha_1-1)! \cdots (\alpha_i+1)!\cdots \alpha_d!}\\
=&\frac{K_1}{\alpha_1!} \cdot \frac{(k+2-\alpha_1)!}{\alpha_2! \cdots (\alpha_i+1)! \cdots  \alpha_d!}+\alpha_1\cdot \frac{(k+1)!}{\alpha_1! \cdots (\alpha_i+1)! \cdots  \alpha_d!}\quad \text{(by induction hypothesis)}\\
=&\frac{(k+1)!}{\alpha_1! \cdots (\alpha_i+1)! \cdots  \alpha_d!} \left(k+2-\alpha_1+\alpha_1\right) \quad (\text{as } K_1 \cdot (k+2-\alpha_1)!=(k+2-\alpha_1)(k+1)!) \\
=&\frac{(k+2)!}{\alpha_1! \cdots (\alpha_i+1)! \cdots  \alpha_d!}.
\end{align*}
\endgroup
Thus the sub-claim is true. From \eqref{induc-eq} we get that 
\[s(s+1)\cdots(s+k)(s+k+1)={(-1)}^{k+2}\sum_{|\beta|=k+1} \frac{(k+2)!}{\beta_1!~\beta_2! \cdots \beta_d!} x^{\beta}\partial^{\beta}.\] The result follows. 
\end{proof}

\s Using notations as in Notation \ref{not}, define the matrix $F = \mathcal{F}(H) = (\mathcal{F}(H_{i,j})) \in R^{dn\times m}$, the $dn\times m$ matrix with entries in $R$ obtained after applying the Fourier transform to each entry of the matrix $H$.
As we have defined $M=D^{m}/N$, in the similar way we define the graded $R$-module $L = R^{m}/(R^{dn}\cdot F)$. Notice the rows of $F$ are homogeneous of degree $\nu_i$ for some $1 \leq i \leq d$.

Consider $S$ as standard graded with $\deg r=0$ for all $r \in R$ and $\deg T_i=1$ for all $i$. Since $\{g_1,\ldots,g_d\}$ is a regular sequence in $S$ so the Koszul complex
{\small
\[\K(g_1,\ldots,g_d): \quad 0 \to S(-d) \xrightarrow{\cdot\left[\begin{smallmatrix}
	(-1)^{d-1}g_d& \cdots &-g_2 & g_1
	\end{smallmatrix}\right]} S(-d+1)^d \to \cdots \to
s(-1)^d \xrightarrow{\cdot\left[\begin{smallmatrix}
	g_1\\
	g_2\\
	\vdots\\
	g_d
	\end{smallmatrix}\right]} S \to  S/(g_1,\ldots,g_d)\to 0.\] }
gives a free resolution of $\sym(I) \cong S/(g_1,\ldots,g_d)$.  Recall that $\bideg g_i=(1, \nu_i)$ for all $i=1,2,3$. Now restricting $\K(g_1,\ldots,g_d)$ to the graded part $p$, we get a free resolution of $\sym_{p}(I)$ as a $R$-module 

{\small
\begin{align*}
0 \to R(-\nu)^{\binom{p}{d}} \to & \bigoplus_{i=0}^{d-1} R(-\nu+\nu_{d-i})^{\binom{p+1}{d}} \to \cdots \to \bigoplus_{i=1}^{d} R(-\nu+\sum_{\substack{j=1 \\ j\neq i}}\nu_j)^{\binom{p+d-1}{d}}\to R^{\binom{p+d}{d}} \to \sym_{p}(I) \to 0
\end{align*}}
( as $\# S_p= \binom{p+(d+1)-1}{(d+1)-1}=\binom{p+d}{d}$ implies that $S_p \cong R^{\binom{p+d}{d}}$ and so on). 
Applying $\Hom_R(-, R)$ we get the complex
{\small
\begin{align*}
\label{res_L}
0 \to R^{\binom{p+d}{d}} \to \bigoplus_{i=1}^{d} R(\nu-\sum_{\substack{j=1 \\ j\neq i}}\nu_j)^{\binom{p+d-1}{d}}\to \cdots \to \bigoplus_{i=0}^{d-1} R(\nu-\nu_{d-i})^{\binom{p+1}{d}} \to R(\nu)^{\binom{p}{d}} \to 0,
\end{align*}}
where the cokernel of the map on the right is the graded $R$-module ${}^*\Ext_R^d(\sym_{p}(I), R)$. Let $L$ denote the cokernel of the map on the right of the new complex induced from the above complex by degree $-\nu$ shifting of each modules appearing in the complex. Therefore
$L(\nu)\cong {}^*\Ext_R^d(\sym_{p}(I), R)$ as graded $R$-modules.

\begin{lemma}
$L$ is a finite length module.
\end{lemma}
\begin{proof}
We have the following commutative diagram 
\begin{align*}
\xymatrix@R0.5cm{R^{dn}\ar@{^{(}->}[d]_{{\mathcal{F}}|_{R^{dn}}} \ar[r]^{\cdot F}& R^m \ar[r]^{\phi}\ar@{^{(}->}[d]_{{\mathcal{F}}|_{R^m}}& \frac{R^m}{R^{dn} \cdot F} \ar[r] \ar[d]^\theta & 0\\
	D^{dn} \ar[r]^{\cdot H}& D^m \ar[r]^{\psi}& \frac{D^m}{D^{dn} \cdot H} \ar[r]  & 0}
\end{align*}

Since $\mathcal{F}|_{R^{m}}(R^{dn} \cdot F) \subseteq D^{dn} \cdot H$, so it induces a map $\theta: L=R^m/(R^{dn}\cdot F) \to D^m/(D^{dn} \cdot H)=M$. We claim that $\theta$ is injective. Let $a \in \ker \theta$. Since $\phi$ is surjective so there exists $b \in R^m$ such that $\phi(b)=a$. Let ${\mathcal{F}}|_{R^m}(b)=c$. Since $\psi(c)=\psi \circ {\mathcal{F}}|_{R^m}(b)=\theta \circ \phi(b)=0$, so $c \in \im (\cdot H)$, that is, there exists some $e \in D^{dn}$ such that $e \cdot H=c$. Notice $e$ is a polynomial in $\partial_i$'s (as all entries of $H$ are in $\partial_i$'s and $c={\mathcal{F}}|_{R^m}(b)$ is a polynomial in $\partial_i$'s). Thus we can construct $e' \in R$ (polynomial in $x_i$'s) such that ${\mathcal{F}}|_{R^{dn}}(e')=e$. Clearly ${\mathcal{F}}|_{R^m}(e' \cdot F)= \mathcal{F}|_{R^{dn}}(e') \cdot \mathcal{F}(F)= e \cdot H=c$. Since $\mathcal{F}|_{R^{m}}$ is injective so we have ${\mathcal{F}}|_{R^m}(e' \cdot F)=b$ and hence $a=0$. The claim follows. 

Thus $L$ is isomorphic to a submodule of $M$. Now $M$ is holonomic by Proposition \ref{M-holo} and so it's length is finite by \cite[2.3, p. 89]{SCC}. Hence length of $L$ is also finite.
\end{proof}


The following result will show that the approximation given in Corollary \ref{divide_b_function} is actually strict. While proving this result we denote the $i$-th component of the free $R$-module $R^m$ by $e^R_i$ and the $i$-th component of the free $D$-module $D^m$ by $e^D_i$ to avoid confusion.
\begin{theorem}[with hypotheses as in \ref{ass}]\label{lowest-x-deg}
		Let $b_M(s)$ be the $b$-function of the holonomic module $M$ defined in Notation \ref{not} and let $u$ be the lowest possible $x$-degree for an element in the graded part $\mathcal{K}_{p,*}$.
		Then \[b_M(s)=s(s+1)\cdots(s+\nu-d-u).\]
\end{theorem}
\begin{proof}
	From Theorem \ref{divide_b_function} we have $s(s+1)\cdots(s+\nu-d-u) \mid b_M(s)$. Recall that $\inn(J_i)\;\cap\;k[s]=(b_{J_i}(s))$ and $b_M(s)= \lcm_{i=1,\ldots,m}(b_{J_i}(s))$. So if we prove that for each $i=1,\ldots,m$ we have 
	$$
	s(s+1)\cdots(s+\nu-d-u) \in \inn(J_i)\;\cap\;k[s],
	$$	
	where $J_i = \pi_i(N \cap D\cdot e^D_i)$, then $b_{J_i}(s) \mid s(s+1)\cdots(s+\nu-d-u)$ and hence $b_{M}(s) \mid s(s+1)\cdots(s+\nu-d-u)$. The result follows. 	
	
	Let $a=\text{end}(L)=\max\{k \mid L_k \neq 0 \}$ (since $L$ is a finite length module), then for any $x_1^{\alpha_1}\cdots x_d^{\alpha_d}$ with $\alpha_1+\cdots +\alpha_d=a+1$ we have 
	\[x_1^{\alpha_1}\cdots x_d^{\alpha_d}e^R_i=(0,\ldots,\underbrace{x_1^{\alpha_1}\cdots x_d^{\alpha_d}}_{i\text{-th}},\ldots,0) \in R^{dn}\cdot F,\]
	where $i=1,\ldots,m$. Applying the inverse of the Fourier transform we get that
	\[\partial_1^{\alpha_1}\cdots\partial_d^{\alpha_d}e^D_i=(0,\ldots,\underbrace{\partial_1^{\alpha_1}\cdots\partial_d^{\alpha_d}}_{i\text{-th}},\ldots,0) \in D^{dn}\cdot H=N,\]Since $\inn(s(s+1)\cdots(s+a))= s(s+1)\cdots(s+a)\in D(\partial_1, \ldots, \partial_d)^{a+1}$ by Lemma \ref{belong} so it follows that
		\[s(s+1)\cdots(s+a) \in \inn(J_i) \;\cap\;k[s]\]
	for each $i=1,\ldots,m$.
	By {\it the local duality theorem for graded modules} (see \cite[Theorem 3.6.19]{BH}) we get 
	$$
	\mathcal{K}_{p,*}=H_{\m}^0(\sym_{p}(I)) \cong {}^*\Hom_{k}\big({}^*\Ext_R^d(\sym_{p}(I), R(-d)), k \big)
	\cong
	{}^*\Hom_{k}\big(L(\nu-d), k \big).
	$$
	Since the grading of ${}^*\Hom_{k}\big(L(\nu-d), k \big)$ is given by 
\[{{}^*\Hom_{k}\big(L(\nu-d), k \big)}_i=\Hom_{i}\big(L(\nu-d), k \big)=\Hom_{k}\big({L(\nu-d-i)}, k \big),\]
	(using notations as in \cite[p. 33]{BH}) we have that $a=\nu-d-u$, and so the statement of theorem follows.
\end{proof}

\section{Computing Hom with duality}
For any $i \geq 0$, we define the $k$-vector space  \[F_i=\{x_1^{\alpha_1}\cdots x_d^{\alpha_d}\partial_1^{\zeta_1}\cdots \partial_d^{\zeta_d}T_1^{\beta_1} \cdots T_{d+1}^{\beta_{d+1}}\mid |\alpha|+|\zeta|+|\beta| \leq i\}\]
with $F_{-1}=0$. Then $\gr(\mathcal{T}) \cong T:= k[x_1,\ldots,x_d,\delta_1,\ldots,\delta_d,T_1,\ldots,T_{d+1}]$ and we get a canonical map $\sigma: \mathcal{T} \to T$ given by $\sigma(x_i)=x_i, \sigma(\partial_i)=\delta_i$ and $\sigma(T_j)=T_j$.
Since $q_1,\ldots,q_d$ are linear on the $T_i$'s and have degree $\nu_1, \ldots, \nu_d$ on the $\partial_i$'s with $\sum_{i=1}^d \nu_i=\nu$, so they are homogeneous polynomials in $T$ having degree $\nu_1+1, \ldots, \nu_d+1$ respectively.

Notice that \cite[Proposition 5.2, Propositoion 5.3]{YC} holds in this case also. Since the Hilbert-Samuel function of $T$ is given by $\binom{t+3d+1}{3d+1}$, thus we have $d(\mathcal{T})=3d+1$. Set $Q=\mathcal{T}/\mathcal{T}(L_1,\ldots,L_d)$. 

\begin{proposition}\label{koszul-d(Q)}
	The following statements hold.
	\begin{enumerate}[{\rm (i)}]
		\item The dimension of $Q$ is $d(Q)=2d+1$. 
		\item The following Koszul complex in $\mathcal{M}_U^{l}(\mathcal{T})$ is exact
		\begin{equation}\label{koszul}
		\mathcal{A}: \quad 0 \to \mathcal{T}(-d) \xrightarrow{\cdot\left[\begin{smallmatrix}
		(-1)^{d-1}L_d & \cdots &-L_2 & L_1
			\end{smallmatrix}\right]}  \mathcal{T}(-d+1)^d \to \cdots \to
	\mathcal{T}(-1)^d	\xrightarrow{\cdot\left[\begin{smallmatrix}
		L_1\\
		L_2\\
		\vdots\\
		L_d
		\end{smallmatrix}\right]} \mathcal{T} \to Q \to 0.
		\end{equation}
	\end{enumerate}
\end{proposition}
\begin{note*} For $d=3$,
	\begin{equation*}
	\mathcal{A} : \quad 0 \to \mathcal{T}(-3) \xrightarrow{\cdot\left[\begin{smallmatrix}
		L_3 &-L_2 & L_1
		\end{smallmatrix}\right]} \mathcal{T}(-2)^3 \xrightarrow{\cdot\left[\begin{smallmatrix}
		-L_2& L_1& 0\\
		-L_3 &0 & L_1\\
		0 &-L_3 & L_2
		\end{smallmatrix}\right]}
	\mathcal{T}(-1)^3 \xrightarrow{\cdot\left[\begin{smallmatrix}
		L_1\\
		L_2\\
		L_3
		\end{smallmatrix}\right]} \mathcal{T} \to Q \to 0.
	\end{equation*}
\end{note*}

\begin{proof}
	(i) Set $\Gamma=\{F_i\}_{i\geq 0}$. Notice $\Gamma'=\{F_i \cap \mathcal{T}(L_1,\ldots,L_d)\}_{i \geq 0}$ and $\Gamma'' =\{F_i/(F_i \cap \mathcal{T}(L_1,\ldots,L_d))\}_{i \geq 0}$ are natural good filtrations for $\mathcal{T}(L_1,\ldots,L_d)$ and $Q$ respectively. By \cite[Lemma 5.1]{SCC} we have the following exact sequence \[0 \to \gr^{\Gamma'}(\mathcal{T}(L_1,\ldots,L_d)) \to \gr^{\Gamma}(\mathcal{T}) \to \gr^{\Gamma''}(Q) \to 0.\] Hence $\gr^{\Gamma''}(Q)= \bigoplus_{i\geq 0} F_i/(F_{i-1}+F_i \cap \mathcal{T}(L_1,\ldots,L_d)) \cong T/(q_1,\ldots,q_d)$.
	
	Let $B=k[x_1,\ldots,x_d,\delta_1,\ldots,\delta_{d+1}]$, $h_i=\sigma(\mathcal{F}(f_i)) \in B$ and $J=(h_1,\ldots,h_{d+1})$ be an ideal in $B$. Note that $\sigma \circ \mathcal{F}|_R: R=k[x_1, \ldots,x_d] \to k[\delta_1, \ldots,\delta_d]=A$ defined by $x_i \mapsto \partial_i \mapsto \delta_i$ is an isomorphism. So applying $\sigma \circ \mathcal{F}|_R$ to \eqref{eq1} we get 
	\begin{equation}\label{Hil-Bur}
	0 \to A^{d} \overset{\varphi'} \longrightarrow A^{d+1} \xlongrightarrow{[h_1, \ldots,h_{d+1}]} J' \to 0,	
	\end{equation}
	where $J'=(h_1,\ldots,h_{d+1})$ an ideal in $A$. Note that $B=A[x_1,\ldots,x_d]$ is a flat extension of $A$. So applying $B \otimes\_$ to \eqref{Hil-Bur} we get a re4solution of $J$, \[0 \to B^{d} \overset{\varphi'} \longrightarrow B^{d+1} \xlongrightarrow{[h_1,\ldots,h_{d+1}]} J \to 0 \] (as $B$ is a flat $A$-module so by \cite[Theorem 7.7]{Mat} we have $J' \otimes_A B= J'B=J$).
	If $\varphi=(a_{ij})_{(d+1) \times d}$ for some $a_{ij}\in R$, then $\varphi'= (\sigma(\mathcal{F}|_R(a_{ij})))_{(d+1) \times d}$. Now \[\sum_{i=1}^{d+1} \sigma \circ \mathcal{F}|_R(a_{ij})T_i=\sigma \circ \mathcal{F}|_S\left(\sum_{i=1}^{d+1} a_{ij}T_i\right)=\sigma \circ \mathcal{F}|_S(g_i)=q_j\] for all $1 \leq j \leq d$, where $\sigma \circ \mathcal{F}|_S: S=k[x_1,\ldots,x_d, T_1, \ldots ,T_{d+1}] \to k[\delta_1, \ldots,\delta_d, T_1, \ldots,T_{d+1}]=U$ is an isomerism defined by $x_i \mapsto \delta_i$ and $T_i \mapsto T_i$. So $[T_1,\ldots,T_{d+1}] \cdot \varphi'=[q_1,\ldots,q_d]$ and hence $\sym(J) \cong T/(q_1,\ldots,q_d)$, where $T=B[T_1,\ldots,T_{d+1}]$. Since $g_1,\ldots,g_d$ is a regular sequence in $S$ so $q_1,\ldots,q_d$ is a regular sequence in $U$ and hence $q_1,\ldots,q_d$ is a regular sequence in $U[x_1,\ldots,x_d]=T$ (as $T$ is a free $U$-module). Hence $\dim T/(q_1,\ldots,q_d)= \dim T-\hgt (q_1,\ldots,q_d)=(3d+1)-d=2d+1$. By \cite[Theorem 13.4]{Mat} it follows that $d(Q)=\dim T/(q_1,\ldots,q_d)=2d+1$.

	(ii) Since $L_i$'s are linear on the $T_j$'s so the shifting of degrees in \eqref{koszul} are clear. Recall that $\mathcal{A}_p \cong \sum_{r=1}^p L_{i_r} e_{i_1 \ldots i_p}$ is a free $\mathcal{T}$-module of rank $\binom{d}{p}$ with basis $\{e_{i_1 \ldots i_p}\mid 1 \leq i_1< i_2 \cdots < i_p \leq d\}$ and the differential $d: K_p \to K_{p-1}$ is defined by $d(e_{i_1 \ldots i_p})=\sum_{i=1}^r (-1)^{r-1}e_{i_1 \ldots \widehat{i_r} \ldots i_p}$ (for $p=1$; $d(e_i)=L_i$). Thus 
	\begin{align*}
	d^2(e_{i_1 \ldots i_{p+1}})= d\left(\sum_{i=1}^r (-1)^{r-1}L_{i_r}e_{i_1 \ldots \widehat{i_r} \ldots i_{p+1}}\right)&=\sum_{r=1}^{p+1} (-1)^{r-1}L_{i_r}\left(\sum_{\substack{s=1\\s \neq r}}^{p+1} (-1)^{s-1}L_{i_s}e_{i_1 \ldots \widehat{i_s} \ldots \widehat{i_r} \ldots i_{p+1}}\right)\\
	&=\sum_{r=1}^{p+1} \sum_{\substack{s=1\\s \neq r}}^{p+1} (-1)^{r+s-2}L_{i_r}L_{i_s}e_{i_1 \ldots \widehat{i_s} \ldots \widehat{i_r} \ldots i_{p+1}}
	\end{align*} 
	(position of $r$ and $s$ may vary). Without loss of generality we may assumed that $s<r$. Then coefficient of $e_{i_1 \ldots \widehat{i_s} \ldots \widehat{i_r} \ldots i_{p+1}}$ in the above expression is $(-1)^{r+s-2}L_{i_r}L_{i_s}+(-1)^{s-1}(-1)^{r-2}L_{i_s}L_{i_r}=(-1)^{r+s-3}(L_{i_s}L_{i_r}-L_{i_r}L_{i_s})$. Although $\mathcal{T}$ is non-commutative, as $L_1,\ldots,L_d$ are defined in the $\partial_r$ and $T_r$ variables so $L_iL_j-L_jL_i=0$ for all $1 \leq i, j \leq d$ with $i \neq j$ and hence $d^2(e_{i_1 \ldots i_{p+1}})=0$ for all $1 \leq p \leq d$. It follows that \eqref{koszul} is a complex. So it is enough to prove exactness in the category $\mathcal{T}$. Now \eqref{koszul} induces the following graded Koszul complex in $T$, 
	\begin{equation*}
	\mathcal{A}' : \quad 0 \to T(-d) \xrightarrow{\cdot\left[\begin{smallmatrix}
		(-1)^{d-1}q_d& \cdots &-q_2 & q_1
		\end{smallmatrix}\right]}  \mathcal{T}(-d+1)^d \to \cdots \to
	\mathcal{T}(-1)^d	\xrightarrow{\cdot\left[\begin{smallmatrix}
	q_1\\
	q_2\\
	\vdots\\
	q_d
	\end{smallmatrix}\right]} \mathcal{T} \to \frac{\mathcal{T}}{(q_1, \ldots, q_d)} \to 0.
	\end{equation*}
	As $q_1,\ldots,q_3$ is a regular sequence, the above complex is exact. Hence by \cite[Lemma 3.13]{Bjork} we get that \eqref{koszul} is exact. 
\end{proof}

\begin{corollary}
	\label{vanishing_Ext}
	For any $j\neq d$ we have {\normalfont ${}^*\Ext_{\mathcal{T}}^j(Q,\mathcal{T})=0$}, and  ${}^*\Ext_{\mathcal{T}}^d(Q,\mathcal{T})\neq 0$.
\end{corollary}
	\begin{proof}
		Since \eqref{koszul} is a free resolution of $Q$ so we have ${}^*\Ext_{\mathcal{T}}^d(Q,\mathcal{T})=0$ for $j>d$. From \cite[Theorem 7.1, p 73]{Bjork} we have $j(Q) + d(Q) = 3d+1$, where $j(Q)=\inf\{k \mid {}^*\Ext_{\mathcal{T}}^k(Q,\mathcal{T})\neq 0\}$. Again $d(Q)=2d+1$ by Proposition \ref{koszul-d(Q)}. Hence $j(Q)=d$. The result follows.
		\end{proof}

\begin{theorem}
	\label{duality}
	For any $i$ we have the following isomorphism of graded $U$-modules {\rm(}with hypotheses as in \ref{ass}{\rm)}
	$$
	{}^*\Ext_{\TT}^i(Q, S) \cong {}^*\Tor_{d-i}^{\TT}\big({}^*\Ext_{\TT}^d(Q, \TT), \;S \big).
	$$
\end{theorem}
\begin{proof}
We have $S \cong \TT/(\partial_1,\ldots, \partial_d)$. Since $\partial_1,\ldots, \partial_d$ is a regular sequence in $\TT$ (considering $\TT$ as left ring), so	a resolution of $S$ in $\mathcal{M}^l_U(\mathcal{T})$ is given by the Koszul complex 
	
	\begin{equation}\label{sec_Koszul}
	\mathcal{B}:=\K(\partial_1, \ldots, \partial_d): \quad 0 \to \mathcal{T} \to \mathcal{T}^d \to \cdots \to  \mathcal{T}^d \to \mathcal{T} \to S \to 0.
	\end{equation}
	
	We define the following third quadrant double complex ${}^*\Hom_{\TT}(\A, \TT) \otimes_\TT \B$,	
	\begin{center}	
		\begin{tikzpicture}
		\matrix (m) [matrix of math nodes,row sep=2em,column sep=3em]
		{{}^*\Hom_{\TT}(\A_d,\TT)\otimes_\TT \B_d &	\ldots & {}^*\Hom_{\TT}(\A_1,\TT)\otimes_\TT \B_d & {}^*\Hom_{\TT}(\A_0, \TT)\otimes_\TT \B_d \\
		\quad \vdots \quad&\quad \vdots \quad & \quad \vdots \quad & \quad \vdots \quad\\
		{}^*\Hom_{\TT}(\A_d,\TT)\otimes_\TT \B_1 &	\cdots & {}^*\Hom_{\TT}(\A_1,\TT)\otimes_\TT \B_1 & {}^*\Hom_{\TT}(\A_0, \TT)\otimes_\TT \B_1 \\
		{}^*\Hom_{\TT}(\A_3,\TT)\otimes_\TT \B_0 &	\cdots & {}^*\Hom_{\TT}(\A_d,\TT)\otimes_\TT \B_0 & {}^*\Hom_{\TT}(\A_0, \TT)\otimes_\TT \B_0. \\
		};
		\path[-stealth]
		(m-1-2) edge (m-1-1)
		(m-1-3) edge (m-1-2)
		(m-1-4) edge (m-1-3)
		(m-2-2) edge (m-2-1)
		(m-2-3) edge (m-2-2)
		(m-2-4) edge (m-2-3)
		(m-3-2) edge (m-3-1)
		(m-3-3) edge (m-3-2)
		(m-3-4) edge (m-3-3)
		(m-4-2) edge (m-4-1)
		(m-4-3) edge (m-4-2)
		(m-4-4) edge (m-4-3)
		(m-1-1) edge (m-2-1)
		(m-2-1) edge (m-3-1)
		(m-3-1) edge (m-4-1)
		(m-1-2) edge (m-2-2)
		(m-2-2) edge (m-3-2)
		(m-3-2) edge (m-4-2)
		(m-1-3) edge (m-2-3)
		(m-2-3) edge (m-3-3)
		(m-3-3) edge (m-4-3)
		(m-1-4) edge (m-2-4)
		(m-2-4) edge (m-3-4)
		(m-3-4) edge (m-4-4)
		;				
		\end{tikzpicture}	
	\end{center}	
Note that this is a double complex in the category of graded $U$-modules, that is, all its elements are graded $U$-modules and all its maps are homogeneous homomorphisms of graded $U$-modules (due to the construction of $\mathcal{M}^l_U(\mathcal{T})$ and $\mathcal{M}^r_U(\mathcal{T})$).
	
	Since each ${}^*\Hom_{\TT}(\A_j, \TT) \in \mathcal{M}^r_U(\mathcal{T})$ is a free module then by computing homology on each column we get that the only the last row does not vanish 
	On the other hand, by \autoref{vanishing_Ext} when we compute homology on each row only the leftmost column does not vanish.
	
	Thus the spectral sequence determined by the first filtration is given by
	\begin{equation*}
	{}^{\text{I}}E_2^{p,q}=\begin{cases}
	{}^*\Ext_{\TT}^p\big(Q,S\big)  \quad \text{if } q = d,\\
	0 \qquad \qquad\quad\text{otherwise,}
	\end{cases}
	\end{equation*}
	and the spectral sequence determined by the second filtration is given by
	\begin{equation*}
	{}^{\text{II}}E_2^{p,q}=\begin{cases}
	{}^*\Tor_{d-q}^{\mathcal{T}}({}^*\Ext_{\mathcal{T}}^d\big(Q,\TT),S\big) \quad \text{if } p = d,\\
	0 \qquad\qquad\qquad\qquad \qquad\text{otherwise.}
	\end{cases}
	\end{equation*}
	
	Since both spectral sequences collapse so we get the following isomorphisms of graded $U$-modules
	\begin{equation*}
	{}^{\text{I}}E_2^{i,d} \cong H^{i+d}(\text{Tot}({}^*\Hom_{\TT}\big(\A, \TT) \otimes_\TT \B)\big) \cong {}^{\text{II}}E_2^{d,i},
	\end{equation*}	
The result follows.
\end{proof}

\begin{theorem}
	\label{deRham_cohomology}[with hypotheses as in \ref{ass}]
	Then we have the following isomorphism of graded $U$-modules
	$$\mathcal{K} \cong H_{dR}^0(Q) = \{w \in Q \mid \partial_i \bullet w= 0 \text{ for all } 1 \leq i \leq d\}.
	$$
	In particular, for any integer $p$ we have an isomorphism of $k$-vector spaces
	$$
	\mathcal{K}_{p,*} \cong H_{dR}^0(Q_{p}) = \{w \in Q_{p} \mid \partial_i \bullet w= 0\text{ for all } 1 \leq i \leq d \}.
	$$
	\begin{proof}
		From the resolution \eqref{koszul} of $Q$ we get the following complex in 
		\begin{equation}
		{}^*\Hom_\TT(\A,\TT):\quad 0 \to \mathcal{T} \xrightarrow{\left[\begin{smallmatrix}
			L_1\\
			L_2\\
			\vdots\\
			L_d
			\end{smallmatrix} \right]\cdot} \mathcal{T}(1)^d \to\cdots \mathcal{T}(d-1)^d\xrightarrow{\left[\begin{smallmatrix}
			(-1)^{d-1}L_d& \cdots &-L_2 & L_1
			\end{smallmatrix}\right] \cdot}
		 \mathcal{T}(d) \to 0.
		\end{equation}
	
		Computing the $d$-th cohomology of this complex we get ${}^*\Ext_{\TT}^d(Q,\TT) \cong \left(\TT/(L_1,\ldots,L_d)\TT\right)(d)$, where  $\TT/(L_1,\ldots,L_d)\TT=\tau(Q)$ is the standard transposition of $Q$.		
		
		Note that the Koszul complex \eqref{sec_Koszul} is a resolution of $S$. Computing the third homology of the Koszul complex $\tau(Q)(d) \otimes_\TT \B$  we get the following isomorphisms of graded $U$-modules
		\begin{align*}
		{}^*\Tor_d^\TT({}^*\Ext_\TT^d(Q,\TT), S) &\cong H_d\big( \tau(Q)(d) \otimes_\TT \B \big)\\
		&\cong
		\{w \in \tau(Q)(d) \mid  w\bullet \partial_i= 0\text{ for all } 1 \leq i \leq d\}.
		\end{align*}
		Since $\tau(T_i)=T_i$ so we have an isomorphism of graded $U$-modules
	\begin{align*}
	&\{w \in \tau(Q)(d) \mid w\bullet \partial_i= 0\text{ for all } 1 \leq i \leq d\}\\
	\cong & \{w \in Q(d) \mid w\bullet \partial_i= 0\text{ for all } 1 \leq i \leq d  \}.
	\end{align*}
		From Proposition \ref{hom} and Theorem \ref{duality} we get the following isomorphisms of graded $U$-modules
		\begingroup
		\allowdisplaybreaks
		\begin{align*}
		\mathcal{K}&\cong {}^*\Hom_\TT(Q, S)(-d)\\
		&\cong{}^*\Tor_d^\TT({}^*\Ext_\TT^d(Q,\TT), S)(-d)\\
		&\cong\{w \in Q \mid w\bullet \partial_i= 0\text{ for all } 1 \leq i \leq d  \},
		\end{align*}
		\endgroup
		The result follows.
	\end{proof} 
\end{theorem}
 \begin{note*} For $d=3$,
 	\begin{equation*}
 	{}^*\Hom_\TT(\A,\TT):\quad 0 \to \mathcal{T} \xrightarrow{\left[\begin{smallmatrix}
 		L_1\\
 		L_2\\
 		L_3
 		\end{smallmatrix} \right]\cdot}\mathcal{T}(1)^3\xrightarrow{\left[\begin{smallmatrix}
 		-L_2& L_1& 0\\
 		-L_3 &0 & L_1\\
 		0 &-L_3 & L_2
 		\end{smallmatrix}\right]\cdot} \mathcal{T}(2)^3\xrightarrow{\left[\begin{smallmatrix}
 		L_3 &-L_2 &L_1
 		\end{smallmatrix}\right]\cdot}  
 	\mathcal{T}(3) \to 0.
 	\end{equation*}
 \end{note*}

\section{Examples and computations}
The following function can be used to compute the $b$-function of each $D$-module $M$ from Notation \ref{not} using {\it Macaulay2}. 

\vspace{0.25cm}
\begin{spacing}{0.73}
	{   \fontfamily{cmtt}
		\selectfont
		\hspace*{-.43cm}\verb|needsPackage "Dmodules"|\\					
		\verb|bFunctionRees = (I, p, d) -> (|\\
		\hspace*{.5cm} \verb|R := ring I;| \\
		\hspace*{.5cm} \verb|W := makeWeylAlgebra R;|\\
		\hspace*{.5cm} \verb|e := d+1;|\\
		\hspace*{.5cm} \verb|V1 := apply(0..(d-1),i->((vars W)_(0,i)));|\\
		\hspace*{.5cm} \verb|D1 := toList V1;|\\
		\hspace*{.5cm} \verb|T := W[S_1..S_e], U := QQ[Z_1..Z_e];|\\
		\hspace*{.5cm} \verb|V2 := apply(1..e,i->S_i);|\\
		\hspace*{.5cm} \verb|D2 := toList V2;|\\
		\hspace*{.5cm} \verb|A := Fourier (map(W, R, D1)) (res I).dd_2;|\\
		\hspace*{.5cm} \verb|L := matrix{D2} * A;|\\   
		\hspace*{.5cm} \verb|src := flatten entries (map(T, U, D2)) basis(p - d, U);|\\
		\hspace*{.5cm} \verb|dest := flatten entries (map(T, U, D2)) basis(p - d+1, U);|\\ 
		\hspace*{.5cm} \verb|m := #src, n := #dest;|\\
		\hspace*{.5cm} \verb|H := mutableMatrix(W, m, d * n);|\\
		\hspace*{.5cm} \verb|V3 := apply(1..e,i->1);|\\
		\hspace*{.5cm} \verb|D3 := toList V3;|\\
		\hspace*{.5cm} \verb|V4 := apply(1..d,i->-1);|\\
		\hspace*{.5cm} \verb|D4 := toList V4;|\\	        
		\hspace*{.5cm} \verb|for i from 0 to m - 1 do (|\\   
		\hspace*{1cm} \verb|M=(i,k)->src#i * L_(0, k-1);|\\
		\hspace*{1cm} \verb|for j from 0 to n - 1 do (|\\	 			
		\hspace*{1.5cm}\verb|for l from 1 to d do (|\\
		\hspace*{2cm}\verb|N=(i,j,k)->M(i,k)// gens ideal(dest#j);|\\
		\hspace*{2cm}\verb|F := N(i,j,l);|\\
		\hspace*{2cm}\verb|H_(i, j + (l-1)*n) = (map(W, T, D3)) F_(0, 0); |\\
	    \hspace*{1.5cm}\verb|);|\\
		\hspace*{1cm} \verb|);|\\
		\hspace*{.5cm} \verb|);|\\
		\hspace*{.5cm} \verb|bM := bFunction(coker matrix H, D4, toList(m:0));|\\	
		\hspace*{.5cm} use R;	\\
		\hspace*{.5cm} bM     \\\
		)\\
	}
\end{spacing}

\begin{example}\label{ex1}
	Let $R=k[x,y,z], S=R[a,b,c,d]$. Consider a matrix \[\varphi=\left[\begin{array}{ccc}
	x & 0 & 0 \\ 
	y& x & 0 \\ 
	z & y & x^2\\ 
	0 & z & z^2
	\end{array}  \right].\] In \cite[Example 4.10]{BM} it was checked that there exists a hight two perfect ideal $I$ satisfying $G_3$ with presentation matrix $\varphi$. It was also shown that $\mathcal{I}= \mathcal{L}+ I_3(B_2(\varphi))$, where $\sym(I)=S/\mathcal{L}$ and 
	
	\begin{equation*}
	B_2(\varphi)=
	\left[\begin{array}{cccc}
	a &b &cx &c(c^2-bd)\\ 
	b & c & 0 & 0 \\ 
	c & d & dz &d(-b^2+ac)
	\end{array} .\right] 
	\end{equation*}
	Thus $\mathcal{K}=\mathcal{I}/\mathcal{L}= I_3(B_2(\varphi))$. Computing in {\it Macaulay2} we get $I_3(\varphi)=(x^4,x^2z^2,-x^3z+xyz^2,-x^2yz+y^2z^2-xz^3)$ and $\mathcal{K}=(c^3x+b^2dz-acdz-bcdx, c^5-b^4d+2ab^2cd-a^2c^2d-2bc^3d+b^2cd^2,   
	b^2c^2dx+c^4dz+ab^2d^2z-a^2cd^2z-abcd^2x-bc^2d^2z, b^3cdx-abc^2dx+bc^3dz-b^2cd^2z)$. Notice in this case $\delta=4$. We also get $b_M(s)=s$ if $p=3,4$ and $b_M(s)=s(s+1)$ if $p=5$. Comparing with Theorem \ref{lowest-x-deg} we get that $s+\delta-3-u=s$, i.e., $u=1$ if $p=3,4$ and $s+\delta-3-u=s+1$, i.e., $u=0$ if $p=5$. Comparing with generators of $\mathcal{K}$ we can say that Theorem \ref{lowest-x-deg} holds true in this case. 
Note that $\mathcal{K}_{*,1}\neq 0$ which satisfies Corollary \ref{vanishing-nonvanishing}.

Computations in Macaulay $2$:	
\begin{verbatim}
i1 : R = QQ[x,y]
o1 = R
o1 : PolynomialRing
i2 : load "bFunctionRees.m2"
i3 : I = ideal(x^4,x^2z^2,-x^3z+xyz^2,-x^2yz+y^2z^2-xz^3)
            4  2 2   3     2   2    2 2   3
o3 = ideal(x ,x z ,-x z+xyz ,-x yz+y z -xz )
o3 : Ideal of R
i4 : for p from 3 to 5 do << factorBFunction bFunctionRees(I, p, 3) << endl;
(s)
(s)
(s)(s + 1)
\end{verbatim}		
\end{example}

\begin{remark}\label{rmk1}
Note that in Example \ref{ex1}, $\mathcal{K}$ is generated by elements of degree $(3,1),(5,0),(5,1)$, so the value of $u$, the lowest possible $x$-degree for an element in the graded part $\mathcal{K}_{p,*}$ will be $1$ for $p=3,4$ and $0$ for $p=5$ onwards which is clear from the output in {\it Macaulay2}. Observing the change of the value of $u$ in the output we can also claim that $\mathcal{K}$ has generators of degrees $(3,1),(5,0)$ which is true. 
\end{remark}

\begin{example}\label{ex2}
	Let $R=k[x,y,z], S=R[K_1,K_2,K_3,K_4]$. Consider a matrix \[\varphi=\left[\begin{array}{ccc}
	x & 0 & 0 \\ 
	y& x & 0 \\ 
	z & y & x^5\\ 
	0 & z & z^5
	\end{array}  \right].\]	Then $I_3(\varphi)=(x^7,x^2z^5,-x^6z+xyz^5,-x^5yz+y^2z^5-xz^6)$. Since $\grade I_3(\varphi)\geq 2$, so by the {\it Hilbert–Burch Theorem} (see \cite[Theorem 1.4.17]{BH}), $I=I_3(\varphi)$ has the free resolution $0 \to R^3 \to R^4 \to I \to 0$ (by the converse part) and $I$ is perfect of $\grade 2$. As $\grade I_1(\varphi) =3$
	$\grade I_2(\varphi) \geq3$ and $\grade I_3(\varphi)\geq 2$,
	by \ref{equv-Gd-Fitt} it follows that $I$ satisfies $G_3$. Again by \cite[Theorem 2.1]{BM} we have $\mathcal{I}=\mathcal{L}:(x,y,z)^6$. So $\mathcal{I}=(xK_2+yK_3+zK_4,xK_1+yK_2+zK_3,x^5K_3+z^5K_4,x^4K_3^3+z^4K_2^2K_4-z^4K_1K_3K_4+x^3yK_3^2K_4+x^3zK_3K_4^2,x^3K_3^5+z^3K_2^4K_4-2z^3K_1K_2^2K_3K_4+z^3K_1^2K_3^2K_4+2x^2yK_3^4K_4+(xy^2+2x^2z)K_3^3K_4^2+2xyzK_3^2K_4^3+xz^2K_3K_4^4,x^2K_3^7+z^2K_2^6K_4-3z^2K_1K_2^4K_3K_4+3z^2K_1^2K_2^2K_3^2K_4-z^2K_1^3K_3^3K_4+3xyK_3^6K_4+(3y^2+3xz)K_3^5K_4^2-y^2K_2K_3^3K_4^3+6yzK_3^4K_4^3-2yzK_2K_3^2K_4^4+3zK_3^3K_4^4-z^2K_2K_3K_4^5, \ldots)$ (computed in {\it Macaulay2}). Thus $\mathcal{K}=(x^4K_3^3+z^4K_2^2K_4-z^4K_1K_3K_4+x^3yK_3^2K_4+x^3zK_3K_4^2,x^3K_3^5+z^3K_2^4K_4-2z^3K_1K_2^2K_3K_4+z^3K_1^2K_3^2K_4+2x^2yK_3^4K_4+(xy^2+2x^2z)K_3^3K_4^2+2xyzK_3^2K_4^3+xz^2K_3K_4^4,x^2K_3^7+z^2K_2^6K_4-3z^2K_1K_2^4K_3K_4+3z^2K_1^2K_2^2K_3^2K_4-z^2K_1^3K_3^3K_4+3xyK_3^6K_4+(3y^2+3xz)K_3^5K_4^2-y^2K_2K_3^3K_4^3+6yzK_3^4K_4^3-2yzK_2K_3^2K_4^4+3zK_3^3K_4^4-z^2K_2K_3K_4^5, \ldots)$ (omitting generators which are linear in $K_i$'s). Note that $\delta=7$. Computing in {\it Macaulay2} we get $b_M(s)=s$ if $p=3,4$; $b_M(s)=s(s+1)$ if $p=5,6$; $b_M(s)=s(s+1)(s+2)$ if $p=7,8$; and $b_M(s)=s(s+1)(s+2)(s+3)$ if $p=9,10$ and $b_M(s)=s(s+1)(s+2)(s+3)(s+4)$ if $p=11$. Comparing with Theorem \ref{lowest-x-deg} we get that $u=4$ if $p=3,4$; $u=3$ if $p=5,6$; $u=2$ if $p=7,8$; $u=1$ if $p=9,10$ and $u=0$ if $p=11$ (so onwards). From the observation in Remark \ref{rmk1} we can also say that $\mathcal{K}$ has generators of degrees $(3,4),(5,3),(7,2),(9,1),(11,0)$ which is in fact true.
	
	Computations in Macaulay $2$ (a little time consuming, but finally we get the result with less effort):
	\begin{verbatim}
	i1 : R = QQ[x,y]
	o1 = R
	o1 : PolynomialRing
	i2 : load "bFunctionRees.m2"
	i3 : I = ideal(x^7,x^2z^5,-x^6z+xyz^5,-x^5yz+y^2z^5-xz^6)
	
	            7  2 5   6     5   5    2 5   6
	o3 = ideal(x ,x z ,-x z+xyz ,-x yz+y z -xz )
	o3 : Ideal of R
	i4 : time for p from 3 to 11 do << factorBFunction bFunctionRees(I, p, 3) << endl;
(s)
(s)
(s)(s + 1)
(s)(s + 1)
(s)(s + 1)(s + 2)
(s)(s + 1)(s + 2)
(s)(s + 1)(s + 2)(s + 3)
(s)(s + 1)(s + 2)(s + 3)
(s)(s + 1)(s + 2)(s + 3)(s + 4)
-- used 726.766 seconds
\end{verbatim}
\end{example}

\section{Applications and Observations}
\s\label{reltype-R(I)} Let $S=\bigoplus_{n \geq 0} S_n$ be a finitely generated standard graded ring over a Noetherian commutative ring $S_0$. Consider $S$
as a factor ring $S_0[T]/\mathcal{J}$ of a polynomial ring corresponding to a minimal generating set of $S_1$.
\begin{definition}\label{reltype}
The maximal degree occurring in a homogeneous minimal generating set of $\mathcal{J}$
	is called {\it the relation type of $S$} and is denoted by $\reltype S$. 
\end{definition}
It is well-known that $\reltype S \leq \reg_{(\underline{T})} S +1$ as an $S_0[T]$-module, see \cite[p. 1]{T}. Consider $S$ as standard graded with $\deg r=0$ for all $r \in R$ and $\deg T_i=1$ for all $i$. Then we can define relation type for $\R(I)=S/\mathcal{I}=R[T_1,\ldots,T_{d+1}]/\mathcal{I}$. Recall that $\mathcal{L}=S\mathcal{I}_{1}$ and $\mathcal{K}_{p} \cong  \mathcal{I}_{p}/(S_{p-1}\mathcal{I}_{1})$ for all $p\geq 1$ (as $\mathcal{K}\cong \mathcal{I}/\mathcal{L}$). Thus if $\reltype \R(I)>1$, then we can say that $\reltype \R(I)$ maximal degree occurring in a homogeneous minimal generating set of $\mathcal{K}$.

In Example \ref{ex2}, $\reltype \R(I) \geq 11$ (in fact equal as we get comparing with the output in {\it Macaulay2} using general method). So in this case we can say that $\reg_{(\underline{T})} \R(I) \geq 11-1=10$. 
On the other hand, we can sometime compute the exact value of $\reltype \R(I)$ using our function when $\reg \R(I)$ is known. 

\begin{remark}
Let $R=k[x_1, \ldots, x_n]$ be a polynomial ring, and $\m=(x_1, \ldots, x_n)$ be the homogeneous maximal ideal of $R$. Let $I$ be an ideal in $R$ generated by homogeneous polynomials of the same degree $\nu$. Let $I_{\m}$ be the image of $I$ in the local ring $(R_{\m}, \m R_{\m})$. As $I$ is a homogeneous ideal, $I$ and $I_m$ have same generating set. Now $C=R \backslash \m$ is a multiplicative set in $\R(I)$ (as $R \subseteq \R(I)$ is a sub-ring). Note that $\R(I_{\m})\cong \R(I) \otimes_R R_{\m} \cong S/\mathcal{I} \otimes_R R_{\m}\cong  C^{-1}S/C^{-1}\mathcal{I} \cong R_{\m}[T_1, \ldots, T_{d+1}]/C^{-1}\mathcal{I}$. So 
	\begin{align*}
	F(I_{\m})=\R(I_m) \otimes_{R_{\m}} R_{\m}/\m R_{\m}\cong \R(I) \otimes_R R_{\m} \otimes_{R_{\m}} R_{\m}/\m R_{\m} &\cong \R(I) \otimes_R R_{\m}/\m R_{\m} \\
	&\cong \R(I) \otimes_R R/\m =F(I),
	\end{align*}
where $F(I)$ is the {\it special fiber ring} of $I$. In view of these relations we can use some properties of $\R(I)$ and $F(I)$ which are known when $I$ is an ideal in a local ring $(R,\m)$.	
\end{remark}

\s\label{CM-ness} Take $R$ and $I$ as in assumption \ref{ass}. Note that in this case {\it second analytic deviation} $\mu(I)-\ell(I)$ one is equivalent to the fact that $\ell(I)=d$, where $\ell(I)=\dim F(I)$ is the {\it analytic spread} of $I$.
\begin{definition}
Let $I=(f_1, \ldots, f_n)=(\underline{f})$ be an ideal of a local ring $R$ of dimension $d$. By $H_i(\underline{f})$ we denote the $i$-th homology of the Koszul complex $\K(\underline{f})$. We say $I$ satisfies sliding depth if \[\depth H_i(\underline{f}) \geq d-n+i \quad \text{for all } i \geq 0.\]
\end{definition}

\begin{definition}
We say $I$ is {\it strongly Cohen-Macaulay} if the Koszul homology modules of $I$ with respect to one (and then to any) generating set are Cohen-Macaulay.
\end{definition}

Since grade two perfect ideals are in the linkage class of a complete intersection so always strongly Cohen-Macaulay if the ring $R$ is Gorenstein local, see \cite[Theorem 1.14, Example 2.1]{H}. Clearly if $I$ is strongly Cohen-Macaulay, then $I$ satisfies sliding depth. By \cite[Proposition 2.4]{UV} we have $\ell(I)=d$ when $(R,\m)$ is a \CM~ local ring of dimension $d$, $I$ satisfies $G_d$, sliding depth, $\mu(I)=d+1$ and $\hgt I \geq 1$. Thus if $(R,\m)$ is a Gorenstein local ring of dimension $d$, $I$ is a grade two perfect ideal satisfying $G_d$ and $\mu(I)=d+1$, then $\ell(I)=d$ and hence $I$ has second analytic deviation one.

Note that \begin{equation}\label{eq-fiiber}
F(I)= \frac{\R(I)}{\m\R(I)} \cong \R(I)_{*,0} \cong \left(\frac{S}{\mathcal{I}}\right)_{*,0}=\frac{S_{*,0}}{\mathcal{I}_{*,0}}=\frac{k[T_1,\ldots,T_{d+1}]}{\mathcal{I}_{*,0}}=\frac{U}{\mathcal{I}_{*,0}}.
\end{equation}
In \cite[3.3]{KPU}, it is observed that $F(I)=k[f_1t, \ldots, f_{d+1}t] \cong k[f_1, \ldots, f_{d+1}] \subset R$ (this isomorphism not necessarily homogeneous). Thus $F(I)$ is a domain (as $R$ is so). Moreover, if we assume that $\dim F(I)=\ell(I)=d$, then $\mathcal{I}_{*,0}$ is a height one prime ideal in $U$ and hence it is principle (as $U$ is a UFD). This fact implies that $F(I)$ is \CM.
%

Clearly $\mathcal{I}_{1,0}=0$. As otherwise, if $a_1T_1+\cdots+a_{d+1}T_{d+1} \in \mathcal{I}_{1,0}$ with $a_i \in k$ such that not all are zero, then $\sum_{i=1}^{d+1} a_i (f_i t)=0$, that is, $\left(\sum_{i=1}^{d+1} a_i f_i\right) t=0$ in $R[t]$. Hence $\sum_{i=1}^{d+1} a_i f_i=0$ in $R[t]$ (as $R[t]$ is a domain) and hence in $R$. This contradicts the fact that $\{f_1, \ldots, f_{d+1}\}$ is a minimal generating set of $I$. Thus $\reltype F(I)=p_0>1$ and $\mathcal{I}_{*,0}$ will be generated by an element $u$ having bi-degree $(p_0,0)$ as an ideal in $U$.
Notice $p_0$ is the minimum value such that $\mathcal{I}_{p_0,0} \neq 0$. Now 
$\mathcal{K}_{p,0} \cong  \mathcal{I}_{p,0}/(S_{p-1,0}~\mathcal{I}_{1,0})$ for all $p \geq 1$ and hence $\mathcal{K}_{*,0} \cong \mathcal{I}_{*,0}$. Clearly image of $u$ in $\sym(I)$ will generate $\mathcal{K}_{*,0}$ and will belong to any minimal generation set of $\mathcal{K}$ (as no other element having bi-degree $(p,q)$ with $q>0$ can generate $\mathcal{K}_{*,0}$). Moreover, if $\dim F(I)=d$, then by \cite[Lemma 5.2]{HK} we have $\reltype F(I)=r(I)+1$, where $r(I)$ denotes the reduction number of $I$. Also $\deg u = p_0=e(F(I))$ (when we consider $S$ is graded) where $e(F(I))$ denotes the Hilbert-Samuel multiplicity of $F(I)$. Since $F(I)\cong k[T_1,\ldots,T_{d+1}]/(u)$ is a hypersurface, so \[\K(u): \quad 0 \to U(-p_0) \overset{\cdot u}{\lrt} U \to 0\] is a minimal free resolution of $F(I)$ and hence $\reg F(I)=p_0-1=\reltype F(I)-1$, where $\reg F(I)$ denotes the regularity of $F(I)$. Thus from the output of our function ({\it bFunctionRees}$(I,p,d)$) in {\it Macaulay2}, we get $\reltype F(I)$ from which we get $e(F(I)), \reg F(I)$ and $r(I)$. 

For Examples \ref{ex1} and \ref{ex2} we get that $\reltype F(I)=5,11$ respectively. 
\begin{remark}
Let $\nu=d$. Since $\sum_{i=1}^d \nu_i=\nu$ so in this case $\nu_i=1$ for all $i$, that is, $I$ has a linear presentation. By Corollary \ref{vanishing-nonvanishing} it follows that for all $p \geq d$, $\mathcal{K}_{p,u} = 0$ for all $u>0$ and $\mathcal{K}_{p,0} \neq 0$. Again by Remark \ref{first-differ} we have $\mathcal{K}_{p,0} = 0$ for all $p<d$. So by the above observation we get that $\reltype F(I)=d$. From the Sub-section \ref{reltype-R(I)} we also get that $\reltype \R(I)=d$.
\end{remark}

\section{Acknowledgments}
I would like to thank Prof. Tony J. Puthenpurakal, my PhD advisor, for his support, guidance, suggesting the problem and for helpful discussions. I am also thankful to Prof. Jugal Verma for his continuous support. 

I thank UGC, Govt. of India for providing financial support for this study.

\end{document}